\documentclass [11pt,a4paper]{article}

\usepackage{amsfonts,amssymb,amsmath,latexsym,makeidx,theorem,color,hyperref,enumitem}
\usepackage[utf8]{inputenc}
\usepackage{enumitem}
\usepackage{anyfontsize}

\DeclareFontShape{OT1}{cmr}{sb}{n}{ <-> ssub * cmr/m/n }{}

\hypersetup{
    colorlinks,
    citecolor=red,
    linkcolor=blue,
    urlcolor=black
}

\newcommand{\gi}{g_i}
\newcommand{\f}{{f_\lambda}}
\newcommand{\F}{\mathfrak{F}}

\newcommand{\de}{\delta}
\newcommand{\R}{\mathbb{R}}
\newcommand{\eps}{\varepsilon}
\newcommand{\ph}{\varphi}
\newcommand{\ra}{\rightarrow}

\newcommand{\fl}{(-\Delta)^\frac{1}{2}}
\newcommand{\N}{\mathbb N}

\newcommand{\w}{\omega_{{\boldsymbol  a}, \boldsymbol \updelta,\boldsymbol\upxi}}
\newcommand{\dis}{\displaystyle}
\newcommand{\Vxi}{\boldsymbol \upxi}
\newcommand{\Vdelta}{\boldsymbol\updelta}
\newcommand{\Va}{\boldsymbol a}
\newcommand{\Vde}{\boldsymbol\updelta}
\newcommand{\wn}{\omega_n}
\newcommand{\pl}{\ph_{\lambda,\Vxi}}

\author{Azahara DelaTorre  \\ \small Mathematisches Institut, Albert-Ludwigs-Universit\"{a}t Freiburg \\ \footnotesize \texttt{azahara.de.la.torre@math.uni-freiburg.de} \and Gabriele Mancini \\ \small Dipartmento SBAI, Sapienza Universit\`a  di Roma \\ \small supported by INdAM - Istituto Nazionale di Alta Matematica\\ \footnotesize \texttt{gabriele.mancini@uniroma1.it}  \and Angela Pistoia \small \\ \small Dipartmento SBAI, Sapienza Universit\`a  di Roma \\ \footnotesize \texttt{angela.pistoia@uniroma1.it}  }

\title{Sign-changing solutions for the one-dimensional non-local sinh-Poisson equation}

\newtheorem{trm}{Theorem}[section]
\newtheorem{prop}[trm]{Proposition}
\newtheorem{cor}[trm]{Corollary}
\newtheorem{lemma}[trm]{Lemma}
\newtheorem{rem}[trm]{Remark}

\newenvironment{proof}{\noindent\emph{Proof.}}{\phantom{ } \hfill$\square$\medskip}
\newenvironment{Si}[1]{\left\{\begin{array}{#1}}{\end{array} \right. }

\usepackage[libertine]{newtxmath}

\DeclareMathOperator{\loc}{loc}

\DeclareMathOperator*{\dist}{dist}

\begin{document}

\maketitle

\begin{abstract}
We study the existence of sign-changing solutions for a non-local version of the sinh-Poisson equation on a bounded one-dimensional interval $I$, under Dirichlet conditions in the exterior of $I$. This model is strictly related to the mathematical description of galvanic corrosion \-phenomena for simple electrochemical systems. By means of the finite-dimensional Lyapunov-Schmidt reduction method,  we construct bubbling families of solutions developing an arbitrarily prescribed number sign-alternating peaks. With a careful analysis of the limit profile of the solutions, we also show that the number of nodal regions coincides with the number of blow-up points. 
\end{abstract}

\maketitle

\section{Introduction}

In this work we consider the  non-local sinh-Poisson equation given by
\begin{equation}\label{Eq}
(-\Delta)^\frac{1}{2} u = \lambda \left( e^{u}-e^{-u} \right)  \ \text{ in } I:=(-1,1), \qquad u\equiv 0 \text{ on } \R\setminus I,
\end{equation}
{with $\lambda\in \R^+$}. This equation is related to mathematical models for the description of galvanic corrosion of a planar electrochemical system consisting of an electrolyte solution and an adjoining metal surface. If the electrolyte is confined in a domain $\Omega\subseteq \R^2$, the electrolytic voltage potential satisfies the non-linear boundary value problem 
\begin{equation}\label{EqGen}
\begin{cases}
\Delta v =0  & \text{ in }\Omega, \\
\frac{\partial v}{\partial \nu} = \lambda \left(e^{\beta v} -e^{-(1-\beta) v}\right) + g  & \text{ on }\partial \Omega,
\end{cases}
\end{equation}
where $\lambda$ and $\beta$ are constants depending on the constituents of the system and $g$ models an externally imposed  current. We refer to \cite{BookCorr,VogXu} for the mathematical derivation of this model, which is due to Butler and Volmer. If one takes $\beta = \frac{1}{2}$ and $g=0$ the problem {becomes equivalent} to 
\begin{equation}\label{EqDomain}
\begin{cases}
\Delta v =0  & \text{ in }\Omega, \\
\frac{\partial v}{\partial \nu} = \lambda \left(e^{v} -e^{-v}\right) & \text{ on }\partial \Omega,
\end{cases}
\end{equation}
Problems \eqref{EqGen}-\eqref{EqDomain}, {and corresponding ones in higher dimension}, have been studied by several authors in recent years (see e.g. \cite{VogXu, KavVog, DDMW,PaPiPiVa,PaPi3D}). Among several generalizations, we mention that  Alessandrini and Sincich \cite{AleSin,AleSin2} have considered  problems of the form 
$$
\begin{cases}
\Delta v =0  & \text{ in }\Omega\subseteq \R^2, \\
\frac{\partial u}{\partial \nu} = \lambda \left(e^{v} -e^{-v}\right) & \text{ on }\Gamma_1,\\
\frac{\partial v}{\partial \nu} = g & \text{ on }\Gamma_2, \\
 v =0   &  \text{ on } \partial\Omega\setminus (\Gamma_1 \cup\Gamma_2),
\end{cases}
$$
where $\Gamma_1$ and $\Gamma_2$ are two open, disjoint portions of $\partial\Omega$. Here $\Gamma_1$ represents the corroded part of $\partial \Omega$, which is not accessible to direct inspection, $\Gamma_2$ is the portion of $\partial \Omega$ where current density can be directly measured, and the remaining part of $\partial \Omega$ is assumed to be grounded.
In this work we want to consider the strictly related problem in which $\Omega = \R^2_+$ is the upper half plane, $\Gamma_1$ is a segment and $\Gamma_2=\emptyset$. Namely, we have 
\begin{equation}\label{EqHP}
\begin{cases}
\Delta v =0  & \text{ in }\Omega, \\
\frac{\partial v}{\partial \nu} = \lambda \left(e^{v} -e^{-v}\right) & \text{ on } I \times \{0\},\\
 v =0 &  \text{ on } (\R\setminus I)\times \{0\}.
\end{cases}
\end{equation}
There is a strict connection between problems \eqref{EqHP} and \eqref{Eq}. Indeed, the Poisson harmonic extension of any solution to \eqref{Eq} solves \eqref{EqHP}. Viceversa, if $v$ solves \eqref{EqHP} and has finite Dirichlet energy, then the boundary trace $u=v(\cdot,0)$ is a solution of \eqref{Eq}.   

Equation \eqref{Eq} can also be considered as a 1-dimensional version of the planar sinh-Poisson problem
\begin{equation}\label{Eq2d}
\begin{cases}
-\Delta u = \lambda (e^{u}-e^{-u}) & \text{ in }\Omega\subseteq \R^2, \\
u= 0 & \text{ on }\partial \Omega,
\end{cases}
\end{equation} 
which arises in the statistical mechanics approach proposed by Onsager \cite{Ons} and Joyce and Montgomery \cite{JM,MJ} to the description of two-dimensional turbulent Euler flows with null total vorticity (we refer to \cite{Chorin, ChoMar, MarPul, Ons} for a physical discussion of this problem).  

\medskip
In recent {years, there has been a great interest} in the construction of sign-changing solutions for problems  \eqref{EqDomain} and \eqref{Eq2d}. When $\Omega $ is the unit disk of $\R^2$, explicit families of solutions to \eqref{EqDomain} were exhibited by Bryan and Vogelius in \cite{BriVog}. As $\lambda\to 0$, such solutions develop an even number of sign-alternating peaks concentrating in separate points of $\partial \Omega$. In \cite{DDMW}, D\'avila, Del Pino,  Musso and Wei proved that solutions of \eqref{EqDomain} with an analogous behavior exist on arbitrary {bounded} domains  with smooth boundary. In fact, for any even $k\in \N$, they constructed two independent branches of solutions developing $k$ sign-alternating peaks on $\partial \Omega$.  In this result, $k$ must be even to ensure the existence of sign-alternating peaks configurations on  $\partial \Omega$, whose connected components are closed curves.  We also refer to  \cite{KavVog,MedVog} for a-priori analysis of blowing-up solutions to \eqref{EqDomain}. 

Concerning problem \eqref{Eq2d}, the existence of sign-changing solutions developing exactly 2 peaks with different sign has been obtained by Bartolucci and Pistoia in \cite{BaPi}. More generally, they proved that if $\xi_1,\ldots,\xi_k\in \Omega$ correspond to a stable critical point of a generalized version of the $k$-point Kirchhoff-Routh path function, then there is a solution $u_\lambda$ of \eqref{Eq2d} such that 
$$
\lambda (e^{u}-e^{-u}) \to 8\pi \sum_{i=1}^k (-1)^i \delta_{\xi_i}
$$
{as $\lambda \to 0$} in the sense of measures, where $\delta_{\xi_i}$ denotes the Dirac delta at $\xi_i$, $i=1,\ldots,k$. The existence of stable critical points of the generalized Kirchhoff-Routh function for $k\ge 3$ was studied by Bartsch, Pistoia and Weth in \cite{Bapiwe}. They give existence results on any domain for $k=3,4$, and on  axially symmetric domains for arbitrary $k\ge 1$. In the latter case, the points $\xi_1,\ldots,\xi_k$ are located on the symmetry axis of $\Omega$ and the corresponding solution of \eqref{Eq2d} develops sign-alternating peaks at $\xi_1,\ldots,\xi_k$.

\medskip
Inspired by these results, the main purpose of this paper is to discuss the existence, for small values of $\lambda$, of  a branch solutions of $\eqref{Eq}$  with an arbitrarily prescribed number of nodal regions. Specifically, for any $k\in \mathbb N$, we will construct a branch of solutions with exactly $k$ sign-alternating peaks in the interval $I=(-1,1)$.

\begin{trm}\label{MainThm}
For any $k\in \N$, there exist $\lambda_{0}>0$ and a family $(u_\lambda)$, defined for $\lambda\in (0,\lambda_0)$, of weak solutions to \eqref{Eq} with exactly $k$ nodal regions. Moreover, for any sequence $(\lambda_n)_{n\in \N} \subset (0,\lambda_0)$ with $\lambda_n \to 0$ as $n\to +\infty$, there exist $\xi_1,\ldots,\xi_k\in I$ with $\xi_1<\xi_2<\ldots<\xi_k$, such that (along a subsequence) we have  
\begin{itemize}
\item $u_{\lambda_n}$ blows-up at $\xi_i$ as $n\to \infty$ for $i=1,\ldots,k$. Namely,  for any $\eps>0$, we have 
\begin{equation}\label{blow-up}
\sup_{(\xi_i-\eps,\xi_i+\eps)} (-1)^{i-1} u_{\lambda_n} \to +\infty, \quad {\text{as }n\to +\infty}. 
\end{equation} 
\item $u_{\lambda_n} \to 2\pi \sum_{i=1}^k(-1)^{i-1} G_{\xi_i}$, in $C^\infty_{loc}(I\setminus \{\xi_1,\ldots,\xi_k\})\cap {C^{0,\frac{1}{2}}_{loc}}(\bar I\setminus\{\xi_1,\ldots,\xi_k\})\quad { \text{as } n\to +\infty}$, where $G_{\xi_i}$ is the Green function for $(-\Delta)^\frac{1}{2}$  with singularity at $\xi_i$ ({explicitly given by} \eqref{FormulaGI}).  
\end{itemize}
\end{trm}

 The solutions provided in Theorem \ref{MainThm} can be considered as the analogue of the ones constructed in \cite{Bapiwe} for problem \eqref{Eq2d} but, working in dimension 1, we are able to obtain a  stronger result and to show that the number of nodal regions coincides with the number of peaks, which doesn't seem to be known for the solutions in \cite{Bapiwe}, except for the case $k=2$ (see \cite{BaPi}).  
 
Our solutions are also strictly related to the solutions of \eqref{EqDomain} with an even number of peaks constructed in \cite{DDMW}. However, here we do not need to impose the evenness of $k$ as the interaction between the first and the last peak is weaker than the interaction with intermediate peaks (unless $k=2$). It should be noted that if $u$ solves \eqref{Eq}, then $-u$ is a solution as well. Thus, Theorem \ref{MainThm} provides two distinct branches of solutions. But, due to the lack of topology of $I$, we cannot expect the existence of other independent branches of solutions with separate peaks as in \cite{DDMW}. Nevertheless, we strongly believe it would be possible to find a different branch of solutions with $k$ nodal regions which develops a  tower of peaks at the origin. Solutions of this kind were constructed for problem \eqref{Eq2d} by Grossi and Pistoia in \cite{GrossiPistoia13} (see also \cite{Pistoia-Ricciardi}).
\medskip
 
Differently from the approach in \cite{DDMW}, we will not rely on {the connection between \eqref{Eq} and} the extended problem \eqref{EqHP}. Instead, the proof of Theorem \ref{MainThm} will be based on the Lyapunov-Schmidt finite-dimensional reduction method, {which has been successfully used to find solutions to \eqref{Eq2d} and other similar problems (see e.g. \cite{BaPi,dkm,egp}). Here, we will apply this technique on the fractional-order Sobolev space  $X^\frac{1}{2}_{0}(I)$, which is defined as the space of all the functions in $H^\frac{1}{2}(\R)$ which  vanish identically outside $I$. A Hilbert structure on $X^\frac{1}{2}_{0}(I)$  is determined by} the scalar product 
\begin{equation}\label{scalar}
<u,v>:= \int_{\R}(-\Delta)^\frac{1}{4} u (-\Delta)^\frac{1}{4} v \, dx, \quad u,v \in X^\frac{1}{2}_0(I),
\end{equation}
with  the corresponding norm given   by 
\begin{equation}\label{norm}
\|u\|:=\| u\|_{X^{1/2}_0}={\|(-\Delta)^\frac{1}{4} u\|_{L^2(\R)}\equiv} \|(-\Delta)^\frac{1}{4} u\|_2,\quad u \in X^\frac{1}{2}_0(I).\end{equation} 
In order to prove Theorem \ref{MainThm}, we proceed as follows. For $\delta>0$ $\xi\in \R$, let us consider the 1-dimensional bubble
\begin{equation}\label{U_dt}
U_{\delta,\xi}(x):= \log\left( \frac{2\delta}{\delta^2+ |x-\xi|^2} \right),
\end{equation}
which solves the fractional-order Liouville equation \cite{DaLMar,HydStructure}
\begin{equation}\label{Liouville}
\fl U_{\de,\xi} = e^{U_{\de,\xi}} \quad \text{ in }\R.
\end{equation}
Note that $U_{\de,\xi}$ blows-up at $\xi$ as $\delta\to 0^+$, namely $\dis{\lim_{\de\to 0^+}U_{\de,\xi}(\xi)=+\infty}$ and $\dis{\lim_{\de\to 0^+}U_{\de,\xi}(x) = -\infty}$, for $x \in \R\setminus\{ \xi \}$. 
%It is possible to correct the bubble $U_{\de,\xi}$ to make it vanish outside $I$.  
For any $\delta>0$, $\xi\in \R$, let $PU_{\de,\xi}$ be the projection of $U_{\de,\xi}$ on $X^\frac{1}{2}_0(I)$, that is 
$$
PU_{\de,\xi}: = (-\Delta)^{-\frac{1}{2}} e^{U_{\de,\xi}},
$$ 
where $(-\Delta)^{-\frac{1}{2}}$ represents the inverse of the  half Laplacian $\fl$. Intuitively, for suitable choices of $\delta$ and $\xi$ we will use $P U_{\de,\xi}$ to model each peak developed by solutions of \eqref{Eq} as $\lambda\to 0$.   Indeed,  Theorem \ref{MainThm}  can be deduced by the following result:

\begin{trm}\label{TrmBetter}
For any $k\in \N$, we can find $\lambda_0>0$ such that, for any $\lambda\in (0,\lambda_0)$, there exist $\delta_{i}=\delta_i(\lambda)\in (0,+\infty)$, $\xi_i = \xi_i(\lambda) \in I$, $i=1,\ldots,k$, and $\ph_\lambda \in X^\frac{1}{2}_0(I)\cap L^\infty(I)$ such that:
\begin{itemize}
\item $\sum_{i=1}^k (-1)^{i-1} PU_{\delta_i,\xi_i} + \ph_\lambda $ is a solution of \eqref{Eq}, for any $\lambda \in (0,\lambda_0)$.
\item  $\displaystyle{\lim_{\lambda\to 0} \delta_i(\lambda) =0}$, for any $i=1,\ldots,k$.
\item  There exists a small $\eta_0\in (0,\frac{2}{k+1})$  (not depending on $\lambda$) such that $-1+\eta_0\le \xi_1 <\ldots<\xi_k \le 1-\eta_0$,  and  $ \displaystyle{\min_{1\le i\le k-1}\xi_{i+1}-\xi_i \ge \eta_0}$, for any $\lambda \in (0,\lambda_0)$.
\item $\|\ph_\lambda \|+ \|\ph_{\lambda}\|_{L^\infty(I)}\to 0 $ as $\lambda \to 0$.
\end{itemize} 
\end{trm} 

This work is organized as follows. In section \ref{Sec:Prelim}, we introduce the notation and state some preliminary results. For the reader's convenience, we also include, in Section \ref{Sec:Out}, an outline of the proof of Theorem \ref{TrmBetter}. The technical aspects of the proof are discussed in  Sections \ref{Sec:E}, \ref{Sec:Lin}, \ref{Sec:Fix} and \ref{Sec:Points}, as we will detail in Section \ref{Sec:Out}. Finally, the conclusion of the proof of Theorem \ref{TrmBetter} and the proof of Theorem \ref{MainThm} are given in Section \ref{Sec:Proofs}.

\section{Notation and preliminary results}\label{Sec:Prelim}
We start by recalling the definition and the main properties of the fractional Laplacian operator. For a given function $u$ in the Schwartz  space $\mathcal{S}$  of rapidly decreasing functions (see e.g. \cite{SteinSBook}), we can define 
$$
(-\Delta)^s u = \mathcal F^{-1}\left(|\xi|^{2s}\mathcal{F}(u)(\xi)\right), \quad s\in (0,1), 
$$
where $\mathcal F$ and $\mathcal F^{-1}$ denote respectively the Fourier and the inverse Fourier transform operators. In fact, this definition makes sense when $|\xi|^{2s}\mathcal{F}(u) \in L^2(\R)$. More generally if $u$ belongs to the space 
\[
L_{s}(\R) :=\left\{ u\in L^1_{loc}(\R)\;:\; \int_{\R}\frac{|u(x)|}{1+|x|^{1+2s}} dx <+\infty\right\},
\]
it is possible to define $(-\Delta)^s$ as the tempered distribution 
$$
<(-\Delta)^s u, \ph> = \int_{\R} u (-\Delta)^s\ph dx, \quad \ph \in \mathcal S. 
$$ 
{
For $s\in (0,1)$, let $H^s(\R)$ denote the fractional-order Bessel potential space
$$
H^s(\R) := \{ u \in L^2(\R)\;:\; |\xi|^{s}\mathcal F(u)\in L^2(\R)\}.
$$ 
This space can be equivalently defined as the space of $L^2$ functions for which the  Gagliardo seminorm is finite (see e.g. \cite[Section 3]{Hitch}). Similarly, $(-\Delta)^s$ can be characterized in terms of singular integrals as 
$$
(-\Delta)^s u(x) = c_{1,s} \, P. V. \int_{\R} \frac{u(x)-u(y)}{|x-y|^{1+2s}} dy,\quad c_{1,s}=\frac{-2^{2s}\Gamma(\frac{1}{2}+s)}{\pi^{\frac{1}{2}}\Gamma(-s)}.
$$
%There is an equivalent definition  $H^s(\R)$ and $(-\Delta)^s$ in terms of singular integrals. Indeed, it can be equivalently defined as the space of $L^2$ functions for which   the squared Gagliardo seminorm 
%$$
%\int_{\R}\int_{\R} \frac{|u(x)-u(y)|^2}{|x-\xi|^{1+2s}}dx 
%$$  
%is finite (see e.g. \cite{Hitch}). 
Throughout the paper,  we will always denote $I:=(-1,1)\subseteq \R$ and  we will consider the space
$$X^\frac{1}{2}_0(I)=\{u\in H^\frac{1}{2}(\R) \: :\;   u\equiv 0 \text{ in }\R\setminus I\},$$ which is a Hilbert space with respect to the scalar product  given in \eqref{scalar}. The corresponding norm will be denoted as in \eqref{norm}.}

 Given $p\ge 1$ and a function $f\in L^p(I)$, we say that $u$ is a weak solution of the problem 
\begin{equation}\label{WeakSol}
\begin{cases}
\fl u = f  & \text{ in I},\\
u = 0 & \text{ in } \R\setminus I,
\end{cases}
\end{equation}
if $u\in X^\frac{1}{2}_{0}(I)$ and {it} satisfies 
$$
\int_{\R}  (-\Delta)^\frac{1}{4}\ph (-\Delta)^\frac{1}{4} u\;    dx = \int_{\R} f \ph \, dx , \quad \forall \ph \in C^\infty_c(I). 
$$

Let also $(-\Delta)^{-\frac{1}{2}}$  represent the inverse of $(-\Delta)^{\frac{1}{2}}$. For any given $p\in (1,+\infty)$, the restriction of $(-\Delta)^{-\frac{1}{2}}$ to $L^p(I)$ is defined by 
\[
\begin{split}
(-\Delta)^{-\frac{1}{2}}: L^p(I)&\rightarrow X^{\frac{1}{2}}_0(I)\\
f&\mapsto u\text { solution to } \eqref{WeakSol}. 
%\left\{
%\begin{split}
%\fl u&=f\quad I\\
%u&=0\quad \R\setminus\{ I\}
%\end{split}\right.
\end{split}.
\]
This operator coincides with the adjoint of the inclusion operator $i_{p'}:X^\frac{1}{2}_0(I)\rightarrow L^{p'}(I)$, where $p'=\frac{p}{p-1}$. In particular, for any $p\in (1,+\infty)$, there exists a constant $C(p)$ such that 
\begin{equation}\label{EllEst}
\|(-\Delta)^{-\frac{1}{2}} f \| \le  C(p) \|f\|_{L^p(I)}, \quad \forall\, f \in L^p(I). 
\end{equation}

The operator $(-\Delta)^{-\frac{1}{2}}$ can also be defined via an explicit representation formula. Throughout the paper, for any $\xi\in I$ we will denote by $G_\xi$ the Green function for $\fl$ on $I$ with singularity at $\xi$, which is given explicitly (see \cite{bucur}) by the formula 
\begin{equation}\label{FormulaGI}
G_\xi(x):= \begin{cases}
\frac{1}{\pi} \log \left(\frac{1-\xi x {+} \sqrt{(1-\xi^2)(1-x^2)}}{|x-\xi|}\right), &   x\in I,
\\ 0,  & x \in \R \setminus I. 
\end{cases}
\end{equation}
We shall often use the notation $G(\xi,x)$ in place of $G_\xi(x)$. 
{For any $f \in L^p(I)$, we have the representation formula 
$$(-\Delta)^{-\frac{1}{2}} f(x)= \int_{I} G_x(y) f(y) dy, \quad x\in I.$$}
We {will} also denote by 
{$H(\xi,x)$ }the regular part of $G_\xi$, namely 
\begin{equation}\label{Robin}
H(\xi,x):= G_{\xi}(x)-\frac{1}{\pi} \log \frac{1}{|x-\xi|}. 
\end{equation}

\subsection{Outline of the proof of Theorem \ref{TrmBetter}: The Lyapunov-Schmidt reduction method}\label{Sec:Out}

For a given $k\in \N$, we fix a small $\eta>0$  with $0<\eta<\frac{2}{k+1}$ and we  define 
\[
\mathcal P_{k,\eta}:=\{\Vxi =(\xi_1,\ldots,\xi_k) \; :\; 1+\xi_1>\eta,\  \xi_k<1-\eta \text{ and } \xi_{i+i}-\xi_i>\eta, i =1,\ldots, k-1 \}.
\]
For $\Va=(a_{1},\ldots,a_{k})
 \in \{-1,1\}^{k}$, $\Vdelta=(\de_1,\ldots,\de_k)\in (0,1)^k$, and $\Vxi = (\xi_1,\ldots,\xi_k)\in \mathcal P_{k,\eta}$,  we denote 
\begin{equation}\label{w}
\w := \sum_{i=1}^k a_{i}  P U_{\de_i,\xi_i}.
\end{equation}
{For the proof of Theorems \ref{MainThm} and \ref{TrmBetter} we could fix $a_i=(-1)^{i-1}$, $1\le i\le k$. But, since many of the estimates  given throughout the paper hold true for a generic choice of the coefficients $a_i$, we will only fix them when it is necessary. }

  Our goal is to find solutions for  \eqref{Eq} of the form
$$
u = \w +\ph, 
$$
where $\ph \in {X}_0^\frac{1}{2}({I}) \cap L^\infty(I)$ is small with respect to both the $X^\frac{1}{2}_0(I)$ and the $L^\infty(I)$ norm.  Throughout the paper we will denote $\f(u):=\lambda \left( e^{u}-e^{-u} \right).$  Then, in terms of $\ph$ equation \eqref{Eq} reads as
\[\begin{split}
\fl \ph &= \f(\w+\ph) -\fl \w \\
& = \underbrace{\f(\w)  -\fl \w}_{=:E}  + \f'(\w) \ph  + \underbrace{\f(\w+\ph) -\f(\w)-\f'(\w)\ph}_{=:N(\ph)},
\end{split}\]
that is 
$$
\fl \ph - \f'(\w) \ph = E + N(\ph).
$$
{It is convenient to rewrite this equation as 
\begin{equation}\label{DefTildeL}
L {\ph}= (-\Delta)^{-\frac{1}{2}} E +  (-\Delta)^{-\frac{1}{2}} N(\ph), \quad \text{ where }\quad  L := (-\Delta)^{-\frac{1}{2}}L=Id-(-\Delta)^{-\frac{1}{2}}\f'(\w). 
\end{equation}} 
For simplicity, here and in the rest of the paper, we will not specify the dependence of $E,N$ and $ L$ on $\lambda, \Va,\Vde$, and $\Vxi$.

 We will prove that there exists a $k-$dimensional subspace $K_{\Vdelta,\Vxi}$ of ${X^\frac{1}{2}_0}(I)$ such that {$ L$} is invertible on $K_{\Vdelta,\Vxi}^\perp$. The space $K_{\Vdelta,\Vxi}$ is spanned by the functions $PZ_{1,i}:= (-\Delta)^{-\frac{1}{2}} e^{U_{\de_i,\xi}}Z_{1,i}$, $i=1,\ldots,k$, where $Z_{1,i}$ is the unique solution which vanishes at infinity of the linearization of \eqref{Liouville} around $U_{\de_i,\xi_i}$ (see Section \ref{Sec:Lin}). 
{Let $\pi : X_0^{\frac{1}{2}} (I)\mapsto K_{\Vde,\Vxi}$ and  $\pi^\perp: X_0^{\frac{1}{2}} (I)\mapsto K_{\Vde,\Vxi}^\perp$  be the projections of $X^\frac{1}{2}_0(I)$ into $K_{\Vde,\Vxi}$ and $K_{\Vde,\Vxi}^\perp$ respectively. Since $X_0^{\frac{1}{2}} (I)=K_{\Vde,\Vxi}\oplus K_{\Vde,\Vxi}^\perp$,   equation \eqref{DefTildeL}} is equivalent to the {following couple of non-linear} problems:
\begin{equation}\label{proj1}
\pi^{\perp} L\ph =  \pi^{\perp}(-\Delta)^{-\frac{1}{2}}  E +  \pi^{\perp}(-\Delta)^{-\frac{1}{2}}  N(\ph).
\end{equation}
\begin{equation}\label{proj2}
\pi L\ph =  \pi(-\Delta)^{-\frac{1}{2}}  E + \pi (-\Delta)^{-\frac{1}{2}}  N(\ph). 
\end{equation}
Exploiting the invertibility of $\pi^\perp  L$ on $K_{\Vde,\Vxi}$, one formulates equation \eqref{proj1} in terms of a fixed point problem for $\ph$. Such problem can be solved if the error term $E$ has small $L^p(I)$ norm for some $p\in (1,2)$, the nonlinear tern $N(\ph)$ decays faster than $\|\ph\|$, and the operator norm of $(\pi^\perp  L)^{-1}$ can be controlled in terms of $\lambda$. We will show that for any choice of $\eta \in (0,\frac{2}{k+1})$, $\Vxi =  (\xi_1,\ldots,\xi_k)\in \mathcal P_{k,\eta}$ and $\lambda$ small enough, these conditions are satisfied by a suitable choice of $\Vdelta =(\delta_1,\ldots, \delta_k)$ depending on $\lambda$ and $\Vxi$. {Specifically}, there exists $\Vdelta = \Vdelta_{\lambda,\xi}\in (0,1)^k$ and $\ph=\ph_{\lambda,\Vxi}\in X^\frac{1}{2}_0(I)$ such that \eqref{proj1} holds. In other words, there exist coefficients $c_i= c_i(\lambda,\xi)$, $i=1,\ldots,k$,  which depend continuously on {$\xi$} such that 
$$
 L \ph_{\lambda,\Vxi} = \pi^\perp (-\Delta)^{-\frac{1}{2}} E +\pi^\perp (-\Delta)^{-\frac{1}{2}} N(\ph_{\lambda,\xi}) = \sum_{j=1}^k c_j  P Z_{1,j}.
$$
Then, $\ph_{\lambda,\xi}$ solves equation \eqref{proj2} if and only if
\begin{equation}\label{System}
c_{i}(\lambda,\Vxi) =0, \quad i =1,\ldots,k. 
\end{equation}
The proof of Theorem \ref{TrmBetter} can be concluded by proving that for any small $\lambda$, there exist $\xi_1, \ldots,\xi_k$ depending on $\lambda$ solving the finite dimensional system \eqref{System}. 

The rest of this paper is organized as follows. In Section \ref{Sec:E} we choose the parameters $\delta_1,\ldots,\delta_k$ and we provide point-wise and $L^p$ estimates on the error terms $E$ and $N$.  Section \ref{Sec:Lin} contains the precise definition of $K_{\Vdelta,\Vxi}$ and the analysis of the invertibility properties of  $\pi^\perp {L}$.  The fix point argument which allows to solve $\eqref{proj1}$ is explained in Section \ref{Sec:Fix}, while system \eqref{System} is studied in Section 6. Finally, we complete the proof of Theorems \ref{MainThm} and \ref{TrmBetter} in Section \ref{Sec:Proofs}. 

\medskip
For the proof described above it is important to point out that all the estimates in Sections \ref{Sec:E}-\ref{Sec:Points} will be uniform with respect to the choice of $\Vxi\in \mathcal P_{k,\eta}$ and of small values of $\lambda$ and $\Vdelta$. For this reason, given two quantities $\Theta_1$, $\Theta_2$ depending on $\lambda, \Vdelta, \Vxi$ and $\eta$ (and eventually other parameters), it is convenient to write $\Theta_1 = O(\Theta_2)$ to indicate $|\Theta_1|\le C \, \Theta_2$, for some constant $C>0$ that does not depend on $\Vxi$, $\Vdelta$ and $\lambda$ (but may depend on $\eta$ and the other parameters, unless otherwise specified). This notation will be used several times throughout the paper.

\section{Choice of the concentration parameters and estimates of the error terms}\label{Sec:E}
 {Let $\w$ be as in \eqref{w}. In order to perform the perturbation argument explained before, we need to be sure that $\w$ is a good approximate solution to \eqref{Eq}. This means we need to estimate the error term}
$$
E= \f(\w)- (-\Delta)^\frac{1}{2}\w,
$$
as defined in Section \ref{Sec:Prelim}. As a first step, we need the following  {lemma}:

\begin{lemma}\label{PU}
For any $\eta>0$, there exists a constant $C_\eta$ such that 
$$
\left|PU_{\delta,\xi} - U_{\de,\xi} + \log(2\delta) - 2\pi H(\xi,\cdot) \right| \le C_\eta \delta^2,
$$
for any $\de\in (0,1), \ \xi \in {(-1+\eta,1-\eta).}$ 
%In particular, in the interval $(\xi-\frac{\eta}{2},\xi+\frac{\eta}{2})$, we have that 
%$$
%PU_{\delta,\xi} = U_{\de,\xi} -\log (2\delta) + 2\pi H(\xi,\xi)+ O(|\cdot-\xi|) + O(\delta^2),
%$$
%uniformly with respect to {$\xi\in (-1+\eta,1-\eta)$} and $\de \in (0,1)$. Moreover,  we have 
%$$
%P U_{\de,\xi} = 2\pi G_{\xi} +O(\delta^2),
%$$
%in $I\setminus ({\xi}-\frac{\eta}{2},{\xi}+\frac{\eta}{2})$, uniformly with respect to {$\xi\in (-1+\eta,1-\eta)$} and $\de \in (0,1)$. 
In particular, we have that 
$$
PU_{\delta,\xi} = \begin{Si}{cl}
U_{\de,\xi} -\log (2\delta) + 2\pi H(\xi,\xi)+ O(|\cdot-\xi|) + O(\delta^2)\;  &  \text{ uniformly in } (\xi-\frac{\eta}{2},\xi+\frac{\eta}{2}),  \\
2\pi G_{\xi} +O(\delta^2) & \text{ uniformly in } I\setminus ({\xi}-\frac{\eta}{2},{\xi}+\frac{\eta}{2}),
\end{Si}
$$
independently by the choice of {$\xi\in (-1+\eta,1-\eta)$} and $\de \in (0,1)$.
\end{lemma}

\begin{proof}
Let $u_{\de,\xi}:=PU_{\delta,\xi} - U_{\de,\xi} + \log(2\delta) - 2\pi H(\xi,\cdot)$. First, we observe that 
$$\fl u_{\de,\xi} =0\quad \text{in }I,$$
since $\fl H(\xi,\cdot) =0$ on $I$  and  $\fl PU_{\de,\xi} =\fl U_{\de,\xi} \text{ in } I$ (by definition of $PU_{\de,\xi}$). {Next, we study the values of $u_{\de,\xi}$ in $\R\setminus I$.} 
Here, since $PU_{\delta,\xi}\in X^\frac{1}{2}_0(I)$ we have $PU_{\delta,\xi}=0$. Then, recalling the expression of $U_{\de,\xi}$ given in \eqref{U_dt}, and noting that  $2\pi H(\xi,x)=2 \log {|x-\xi|}$ in $\R\setminus I$, we find that 
$$u_{\de,\xi}(x)=- U_{\de,\xi}(x) + \log(2\delta) -2\log {|x-\xi|} = \log (\de^2+|x-\xi|^2)-2\log |x-\xi|.$$
{Since $x\in \R\setminus I$ and $\xi\in (-1+\eta,1-\eta)$, we get $|x-\xi|\ge \eta$. Then, we} can find  $C_\eta$ s.t. $|u_{\de,\xi}|\le C_\eta \de^2$ in $\R\setminus I$. 
%In summary, our function $u_{\de,\xi}$ satisfies 
%\begin{equation}
%\begin{split}
%\left\{
%\begin{split}
%\fl u_{\de,\xi}=0  &\quad \text{ in }  I,\\
%|u_{\de,\xi}|\le C_\eta \de^2  &\quad \text{ in } \R\setminus I.
%\end{split}\right.
%\end{split}
%\end{equation}
By the Maximum principle (see Lemma 6 in \cite{Servadei_Valdinoci14}) we have the desired result.
\end{proof}

{Next,} we shall fix $\delta_1,\ldots, \delta_k$, in order to make the error term $E$ small near each of the points $\xi_1,\ldots,\xi_k$.   Note that, {for any $1\le i\le k$,} in the interval $(\xi_i -\frac{\eta}{2},\xi_{i}+\frac{\eta}{2})$ we have the uniform expansion  
\begin{equation}\label{ContoE}\begin{split}
E & = \lambda e^{\sum_{j=1}^k a_j P U_{\de_j,\xi_j}} - \lambda e^{-\sum_{j=1}^k a_j P U_{\de_j,\xi_j}} -\sum_{j=1}^k a_j e^{U_{\de_j,\xi_j}} \\
& = \frac{\lambda}{(2\de_ i)^{a_i}} e^{a_i  U_{\de_i,\xi_i} +2\pi a_i H(\xi_i,\xi_i) + 2\pi  \sum_{j\neq i} a_j G_{\xi_j}+\sum_{i=1}^kO(\de_i^2) + O(|\cdot -\xi_i|)} \\ &\quad - \lambda (2\de_ i)^{a_i} e^{-a_i  U_{\de_i,\xi_i} -2\pi a_i H(\xi_i,\xi_i) -2\pi  \sum_{j\neq i} a_j G_{\xi_j}+\sum_{i=1}^k O(\de_i^2)+O(|\cdot-\xi_i|)} - a_i e^{U_{\de_i,\xi_i}} -\sum_{j\neq i} a_j e^{U_{\de_j,\xi_j}}\\
& =  \frac{\lambda a_i}{2\de_i}  e^{U_{\de_i,\xi_i}} e^{2\pi  H(\xi_i,\xi_i) +2\pi a_i\sum_{j\neq i} a_j G_{\xi_j}+O(|\Vde|^2)+O(\cdot-\xi_i|)}+ O( \lambda \delta_i e^{-U_{\de_i,\xi_i}}) - a_i e^{U_{\de_i,\xi_i}} -\sum_{j\neq i} a_j e^{U_{\de_j,\xi_j}},
\end{split} 
\end{equation}
where we have used Lemma \ref{PU} and $a_i\in \{-1,1\}$, $i=1,\ldots,k$. 
Moreover, we have that 
\begin{equation}\label{-U}
\de_i e^{-U_{\de_i,\xi_i}{(x)}} = \frac{\de_i^2+|x-\xi_i|^2}{2} = O(\de_i^2) + O(|x-\xi_i|^2) 
\end{equation}
and, for $j\neq i$,  that 
$$
e^{U_{\de_j,\xi_j}(x)} =  \frac{2\de_j}{\de_j^2+|x-\xi_j|^2} = \frac{2\de_j}{|x-\xi_j|^2} + O(\de_j^3)  = \frac{2\de_j}{|\xi_i-\xi_j|^2} +O(\de_j|x-\xi_i|) +  O(\de_j^3) = O(\de_j).
$$

For $i\in\{1,\ldots,k\}$, let us consider the functions 
\begin{equation}\label{F_i}
F_i(\Vxi):= 2\pi  H(\xi_i,\xi_i) +2\pi a_i\sum_{j\neq i} a_j G_{\xi_j}(\xi_i),
\end{equation}
so that  estimate \eqref{ContoE} rewrites as 
\begin{equation}\label{FirstE}\begin{aligned}
E  & = a_i e^{U_{\de_i,\xi_i}} \left( \frac{\lambda }{2\de_i}  e^{F_i(\Vxi)+O(|\Vde|^2)+O(|\cdot-\xi_i|)} -1\right)+O(\lambda)+\sum_{j\neq i} O(\de_j).
\end{aligned}
\end{equation}

In order to make the main term of the {above expansion} small, we choose
\begin{equation}\label{delta}
\delta_i =\delta_i(\lambda,\Vxi):=  \frac{\lambda}{2} e^{F_i(\Vxi)}= \frac{\lambda}{2} e^{2\pi  H(\xi_i,\xi_i) +2\pi a_i \sum_{j\neq i} a_j G(\xi_i,\xi_j)}, \quad i =1,\ldots,k.
\end{equation}

\medskip
%\begin{prop}\label{E precise} Let $\Vdelta = (\delta_1,\ldots,\delta_k)$ be given by \eqref{delta}, {Then, for $i=1,\ldots,k$,  we have} 
%$$
%{E(x)}= a_i F_i'(\xi_i) e^{U_{\de_i,\xi_i}}(x-\xi_i)+  {O\left(e^{U_{\de_i,\xi_i}(x)} (|x-\xi_i|^2+\lambda^2 )\right), } \; 
%$$
%for  $x\in (\xi_i-\frac{\eta}{2},\xi_i+\frac{\eta}{2})$, {uniformly with respect to the choice of $\lambda\in(0,1)$ and $\Vxi \in \mathcal P_{k,\eta}$.}
%\end{prop}
%\begin{proof}
%Thanks to \eqref{delta}, we have $\delta_i=O(\lambda)$, $i=1,\ldots,k$, uniformly for $\Vxi\in \mathcal P_{k,\eta}$. Thus \eqref{FirstE} yields
%{
%$$\begin{aligned}
%E (x) & = a_i e^{U_{\de_i,\xi_i}(x)} \left( e^{F_i(x)-F_i(\xi_i)+O(|\Vde|^2)} -1\right)+O(\lambda)\\
%&   = a_i e^{U_{\de_i,\xi_i}(x)} F'_i(\xi_i)(x-\xi_i) + a_i e^{U_{\de_i,\xi_i}} (O(|x-\xi_i|^2)+O(\lambda^2))+ O(\lambda),
%\end{aligned}
%$$
%for any $x\in  (\xi_i-\frac{\eta}{2},\xi_i+\frac{\eta}{2})$.} Estimating $e^{-U_{\de_i,\xi_i}}$ as in \eqref{-U}, we observe that 
%\begin{equation}\label{EstLambda}
%O(\lambda) ={ e^{U_{\de_i,\xi_i}(x)}  O(\lambda e^{- U_{\de_i,\xi_i}(x)}) = e^{U_{\de_i,\xi_i}(x)}} \left(O(\lambda^2)+O(|x-\xi|^2)\right), 
%\end{equation}
%and the conclusion follows.
%\end{proof}
%
%The precise expansion of Proposition \ref{E precise} { will only be used, in Proposition \ref{main_prop}, to choose the location of the points $\xi_1,\ldots,\xi_n$.} In the rest of the paper, it will be sufficient to know the following weaker integral estimate. 

With this choice, we get the following integral estimate on $E$.

\begin{lemma}\label{Est_E} { Let $\Vdelta = (\delta_1,\ldots,\delta_k)$ be as in \eqref{delta}.} For any $p\in (1,\infty)$ one has
$$
\|E\|_{L^p(I)} = O(\lambda^\frac{1}{p}), 
$$
{uniformly with respect to the choice of $\lambda\in (0,1)$ and $\Vxi =(\xi_1,\ldots,\xi_k)\in \mathcal P_{k,\eta}$. }
\end{lemma}
\begin{proof}
Thanks to \eqref{delta}, we have $\delta_i=O(\lambda)$, $i=1,\ldots,k$ uniformly for $\Vxi\in \mathcal P_{k,\eta}$. Then,   \eqref{FirstE} and \eqref{delta}  yield 
\[
\begin{aligned}
E(x) &= a_i e^{U_{\de_i,\xi_i}}\left(O(\lambda^2)+O(|x-\xi_i|)\right) + O(\lambda)\\ & = O\left(e^{U_{\de_i,\xi_i}} |x-\xi_i| \right) + O(\lambda) ,
\end{aligned}
\]
uniformly in $(\xi_i-\frac{\eta}{2},\xi_i+\frac{\eta}{2})$, $i=1,\ldots,k$, where the last equality follows by 
$e^{U_{\de_i,\xi_i}}\le \frac{2}{\de_i}=O(\lambda^{-1}).$ 
Moreover, using Lemma \ref{PU}, we get 
$$
E = \lambda e^{2\pi\sum_{j=1}^k a_j G_{\xi_j} + O(\lambda^2)}- \lambda e^{-2\pi\sum_{j=1}^k a_j G_{\xi_j} + O(\lambda^2)} - \sum_{j=1}^k a_j e^{U_{\de_j,\xi_j}}  = O(\lambda),
$$
uniformly in ${I}\setminus \cup_{i=1}^k (\xi_i - \frac{\eta}{2},\xi_i+\frac{\eta}{2})$.
Using these estimates, we can assert that 
\begin{equation*}
\begin{split}
\|E\|^p_{L^p(I)}=%\int_I E^p\,dx=
& \sum^k_{i=1}\int^{\xi_i +\frac{\eta}{2}}_{\xi_i-\frac{\eta}{2}} \left( O(e^{U_{\de_i,\xi_i}}  |x-\xi_i|)  +O(\lambda)\right)^p\,dx+\int_{I\setminus \cup_{i=1}^k (\xi_i-\frac{\eta}{2},\xi_i +\frac{\eta}{2})}O(\lambda)^p\,dx\\
=&  \sum^k_{i=1}\int^{\xi_i +\frac{\eta}{2}}_{\xi_i-\frac{\eta}{2}} \left( O(e^{U_{\de_i,\xi_i}}  |x-\xi_i|) \right)^p\,dx  + O(\lambda)^p.
\end{split}
\end{equation*}{
Since $p>1$, with the change of variable $y=\frac{|x-\xi_i|}{\delta_i}$, we find 
$$
\begin{aligned}
\int^{\xi_i +\frac{\eta}{2}}_{\xi_i-\frac{\eta}{2}} \left( e^{U_{\de_i,\xi_i}}  |x-\xi_i| \right)^p dx =\int^{\xi_i +\frac{\eta}{2}}_{\xi_i-\frac{\eta}{2}} \left( \frac{2\delta_i|x-\xi_i|}{\delta_i^2+ |x-\xi_i|^2}  \right)^p\,dx = \de_i \int^{\frac{\eta}{2\delta_i}}_{-\frac{\eta}{2\delta_i}} \left( \frac{2 y}{1+ y^2}  \right)^p dy =O(\de_i) =O(\lambda),
\end{aligned}
$$
for  $i=1,\ldots,k$. We can so conclude that 
$$
\|E\|^p_{L^p(I)} = O(\lambda) + O(\lambda^p)=O(\lambda), \quad{ \forall \ p>1}
.$$}
\end{proof}

\begin{rem}
Using the change of variable of the proof above, one can easily verify that, for any $p,q\ge 0$, the following useful estimate holds {as $\delta\to 0$, uniformly with respect to $\xi\in \R$:}
\begin{equation}\label{gen Est}
\int_{\xi-\frac{\eta}{2}}^{\xi+\frac{\eta}{2}} e^{p\, U_{\de,\xi}}|x-\xi|^q dx =  \begin{Si}{cc}
O(\de^{q-p+1}) & \text{ if } 2p-q>1,\\
O(\de^p |\log \de|) & \text{ if } 2p-q =1,\\
O(\de^{p})  & \text{ if } 2p-q<1.
\end{Si}
\end{equation}
\end{rem}

\begin{rem}
For $p=1$, the argument of Lemma \ref{Est_E}  gives 
\[
\|E\|_{L^1(I)} = O(\lambda|\log \lambda|).
\]
\end{rem}

\subsection{Estimates on the non-linear error term}\label{Sec:N}
In this section we look for estimates on the non-linear error term 
$$
N(\ph) = \f(\w+\ph)-\f(\w ) - \f'(\w)\ph,
$$
as defined in Section \ref{Sec:Prelim}. The following lemma shows that $N$ depends quadratically on $\ph$.

\begin{lemma}\label{est_L} Let  $\Vdelta = (\delta_1,\ldots,\delta_k)$ be as in \eqref{delta}.  For any $p\geq1$ and $s>p$, there exists a constant $C_{p,s,\eta}>0$, depending only on $p$, $s$ and $\eta$, such that 
$$
\|N(\ph_1)-N(\ph_2)\|_{L^p(I)} \le C_{p,s,\eta} \lambda^{\frac{1}{s}-1}  \|\ph_1-\ph_2\|(\|\ph_1\|+\|\ph_2\|),
$$
for any {$\lambda\in (0,1)$, $\Vxi\in \mathcal P_{k,\eta}$} and $\ph_1,\ph_2 \in X^\frac{1}{2}_0(I)$ satisfying $\|\ph_1\|,\|\ph_2\|\le 1$.
\end{lemma}
\begin{proof}
First of all, we observe that for any $x\in I$, there exist ${t_1=}t_1(x),\ {t_2=}t_2(x)\in [0,1]$ such that 
$$\begin{aligned}
N(\ph_1)-N(\ph_2)  & = \f(\w+\ph_1) - \f(\w+\ph_2) - \f'(\w) \ph_1 + \f'(\w) \ph_2  \\
&=  \f'(\w +t_1  \ph_1 + (1-t_1)\ph_2) (\ph_1-\ph_2) - \f'(\w) (\ph_1-\ph_2)   \\
& = \f''(\w +t_2 t_1 \ph_1 +t_2 (1-t_1)\ph_2 )(\ph_1-\ph_2)(t_1 \ph_1+(1-t_1) \ph_2).
\end{aligned}
$$
Denoting {$\tau_1 = t_1 t_2\in[0,1]$} and {$\tau_2= t_2(1-t_1)\in[0,1]$}, we get that 
$$
\begin{aligned}
|N(\ph_1)-N(\ph_2)|& \le |\f''(\w+{\tau_1}\ph_1+{\tau_2}\ph_2)||\ph_1-\ph_2||t_1\ph_1+(1-t_1)\ph_2|\\
&\le |\f''(\w+ {\tau_1} \ph_1+{\tau_2}\ph_2)||\ph_1-\ph_2|(|\ph_1|+|\ph_2|). 
\end{aligned}
$$
Noting that $\f'' = \f$ and that $|\f(t)|\le 2 \lambda e^{|t|}$, we get 
\begin{equation}\begin{aligned}\label{NDiff}
|N(\ph_1)-N(\ph_2)| \le 2\lambda e^{|\w|+|\ph_1|+|\ph_2|} |\ph_1-\ph_2| (|\ph_1|+|\ph_2|)
& \le 2 \lambda e^{|\w| + {\ph_3}} |\ph_1-\ph_2|\ph_3,
\end{aligned}
\end{equation}
where $\ph_3 = |\ph_1|+|\ph_2|$. 
{Additionally, for any choice of $s>p\ge1$, we can find  $s_1,s_2,s_3> 1$ such that $\frac{1}{s_1}+\frac{1}{s_2}+\frac{1}{s_3}+\frac{1}{s}= \frac{1}{p}$. Then, H\"older's inequality implies that 
\begin{equation}\label{multiHolder}
\|\lambda e^{|\w|} e^{|\ph_3|} |\ph_1-\ph_2| \ph_3 \|_{L^p(I)} \le \|\lambda e^{\w}\|_{L^s(I)} \|e^{\ph_3}\|_{L^{s_1}(I)}  \|\ph_1-\ph_2\|_{L^{s_2}(I)} \|\ph_3\|_{L^{s_3}(I)}.
\end{equation}
Now, using Lemma \ref{PU}, we see that $\lambda e^{|\w|} = O(\lambda)$ in $I \setminus \cup_{i=1}^k(\xi_i-\frac{\eta}{2},\xi_i+\frac{\eta}{2})$, and that 
$$
\lambda e^{|\w|} = \lambda e^{PU_{\de_i,\xi}+O(1)} = O(e^{U_{\de_i,\xi_i}}) 
$$
in $(\xi_i-\frac{\eta}{2},\xi_i+\frac{\eta}{2})$ for $i=1,\ldots,k$. Therefore 
\begin{equation}\label{Eq s}
\| \lambda e^{|\w|} \|_{L^{s}(I)}^{s} =  \sum_{i=1}^k \int_{\xi_i-\frac{\eta}{2}}^{\xi_i+\frac{\eta}{2}} O\left(e^{s U_{\de_i,\xi_i}} \right)dx  +O(\lambda^s) = O(\lambda^{1-s})+ O(\lambda^s) = O(\lambda^{1-s}),  
\end{equation}
where we used \eqref{gen Est} and  $1-s<s$. Note that the quantity $O(\lambda^{1-s})$ depends on $\eta$ and $s$. 

Using the Moser-Trudinger inequality (see \cite{MartinazziMT}), we get that 
\begin{equation}\label{Eq s1}
\int_{\R} e^{s_1 |\ph_3|} dx \le e^{\frac{s_1^2}{4\pi}\|\ph_3\|^2} \int_I e^{\pi \frac{\ph_3^2}{\|\ph_3\|^2}} dx \le C e^{\frac{s_1^2}{4\pi} \|\ph_3\|^2}\le C(s_1).  
\end{equation}
Finally, thanks to Sobolev's inequality, we have the estimates
\begin{equation}\label{Eq s2s3}\|\ph_1-\ph_2\|_{L^{s_2}(I)} \le C(s_2) \|\ph_1-\ph_2\|\quad \text{ and }\quad \|\ph_3\|_{L^{s_3}(I)} \le C(s_3) \|\ph_3\| \le C(s_3) (\|\ph_1\| +\|\ph_2\|).
\end{equation}
Thus, replacing  \eqref{multiHolder}-\eqref{Eq s2s3} into  \eqref{NDiff}, we obtain
$$
\|N(\ph_1)-N(\ph_2)\|_{L^{p}(I)} \le C \lambda^{\frac{1-s}{s}} \|\ph_1-\ph_2\| ( \|\ph_1\|+\|\ph_2\| ),
$$
with $C$ depending only on $\eta, s, s_1,s_2$ and $s_3$. Since the choice of $s_1,s_2$ and $s_3$ depends only on $s$ and $p$, we get the conclusion. }
\end{proof}

\begin{rem}\label{similar}
Repeating the argument of the above proof, we can show that, for any $s,s_1>p\ge 1$ such that $\frac{1}{s_1}+\frac{1}{s}<\frac{1}{p}$, there exists a constant $C=C(p,s,s_1,\eta)$ such that 
$$
\|\f(\w +\ph)-\f(\w)\|_{L^p(I)}  + \|\f'(\w +\ph)-\f'(\w)\|_{L^p(I)} \le C  \lambda^{\frac{1-s}{s}} e^{\frac{s_1^2}{4\pi}\|\ph\|^2} \|\ph\|,
$$
%and 
%$$
%\|\f'(\w +\ph_1)-\f'(\w +\ph_2)\|_{L^p(I)}  \le C  \lambda^{\frac{1-s}{s}} e^{\frac{s_1^2}{4\pi}\||\ph_1|+|\ph_2|\|} \|\ph_1-\ph_2\|,
%$$
for any $\ph \in X_0^\frac{1}{2}(I)$, $\Vxi \in \mathcal P_{k,\eta}$ and $\lambda\in (0,1)$.   
{Note that whenever $s>p$, it is possible to choose $s_1>p$ large enough so that $\frac{1}{s_1}+\frac{1}{s} <\frac{1}{p}$.}
\end{rem}

\section{Properties of the linearized operator}\label{Sec:Lin}
This section is devoted to the study of the linear operator $ L:X_0^\frac{1}{2}(I) \to X_0^\frac{1}{2}(I)$, defined as
$$
 L \ph =  \ph - (-\Delta)^{-\frac{1}{2}} \f'(\w)\ph.
$$ In particular, we are interested in exhibiting an approximate kernel  of $ L$.  As a first step we describe the behavior of the term $\f'(\w)$.

\begin{lemma}\label{ExpDer}
For any $i= 1,\ldots,k$, we have the expansion 
$$
\f'(\w ) = \begin{Si}{cl} 
e^{U_{\de_i,\xi_i}} (1+O(|\cdot-\xi_i|) +O(\lambda^2)) & \text{ uniformly in }  (\xi_i-\frac{\eta}{2},\xi_i + \frac{\eta}{2}),\\
O(\lambda)  &  \text{ uniformly in } I \setminus \cup_{i=1}^k (\xi_i-\frac{\eta}{2},\xi_i + \frac{\eta}{2}).
\end{Si}
$$ 
In particular, $\|\f'(\w)\|_{L^1(I)}=O(1)$.  
\end{lemma}
\begin{proof}
Indeed, arguing as in \eqref{ContoE}, we get 
\[\begin{split}
\f'(\w)
& = \frac{\lambda}{(2\de_ i)^{a_i}} e^{a_i  U_{\de_i,\xi_i} +2\pi a_i H(\xi_i,\cdot) + 2\pi  \sum_{j\neq i} a_j G_{\xi_j}+O(\lambda^2)}  + \lambda (2\de_ i)^{a_i} e^{-a_i  U_{\de_i,\xi_i} -2\pi a_i H(\xi_i,\cdot) -2\pi  \sum_{j\neq i} a_j G_{\xi_j}+O(\lambda^2)}\\
& = \frac{\lambda}{2\de_ i} e^{  U_{\de_i,\xi_i} +{F_i(\Vxi)} +O({|\cdot-\xi_i|}) +O(\lambda^2)} + O(\lambda \delta_i e^{-U_{\de_i,\xi_i}})\\
& = e^{U_{\de_i,\xi_i}+O({|\cdot-\xi_i|})+O(\lambda^2)} + O(\lambda)\\
& = e^{U_{\de_i,\xi}} (1+O({|\cdot-\xi_i|})+O(\lambda^2)) + O(\lambda) \\
& {=e^{U_{\de_i,\xi_i}} (1+O({|\cdot-\xi_i|}) +O(\lambda^2)),}
\end{split}\]
in $(\xi_i-\frac{\eta}{2},\xi_i + \frac{\eta}{2})$, { where in the last equality we used \eqref{-U} to estimate $O(\lambda)$ as  
\[
O(\lambda) ={ e^{U_{\de_i,\xi_i}(x)}  O(\lambda e^{- U_{\de_i,\xi_i}(x)}) = e^{U_{\de_i,\xi_i}(x)}} \left(O(\lambda^2)+O(|x-\xi|^2)\right).
\]
Moreover,} $\f'(\w)=O(\lambda)$ in $I \setminus \cup_{i=1}^k (\xi_i-\frac{\eta}{2},\xi_i + \frac{\eta}{2})$, since on this set we have
$$
\w = 2\pi \sum_{j=1}^k a_j  G_{\xi_j} + O(\lambda^2) = O(1) . 
$$
\end{proof}

Next, we focus on the kernel of $ L$. Observe that if $\ph\in X_0^\frac{1}{2}(I)$ and ${ L\ph} =0$ then, the scaled functions $\Phi_i(x):= \ph({ \xi_i+}\de_ix)$ are weak solutions to
$$
(-\Delta)^\frac{1}{2}\Phi_i + \de_i \f'(\w({\xi_i}+\de_i \cdot))=0,\quad i=1,\ldots,k, 
$$ 
in the expanding intervals $(\frac{-\xi_i-1}{\delta_i},\frac{1-\xi_i}{\delta_i})$. According to Lemma \ref{ExpDer}, for any fixed $y\in \R$, we have  
$$
\de_i \f'(\w({\xi_i}+\de_i y)) \sim \de_i e^{U_{\de,\xi_i}(\de_i y)}  = \frac{2}{1+y^2}. 
$$
Then, $\Phi_i$ should behave locally as a solution of the problem 
\begin{equation}\label{LiouvilleScaled}
(-\Delta)^\frac{1}{2}\Phi = \frac{2\Phi}{1+|\cdot|^2}   \quad \text{in } \R. 
\end{equation}
This equation was studied by many authors. In particular Santra \cite[Theorem 1.4]{Santra} (see also \cite{DDS}) proved that the only bounded solutions to  \eqref{Liouville} are {linear combinations of} the functions 
\begin{equation}\label{ZsuR}
Z_{0}(y):= \frac{1-y^2}{1+y^2} \quad \text{ and } \quad Z_{1}(y):= \frac{2 y}{1+y^2}. 
\end{equation}

Here we will need a small modification of this classification result.  Let us consider the spaces
\begin{equation}\label{spaces}
\begin{aligned}
\mathcal L :=\{u \in L^1_{loc}(\R)\;:\; |u|^2(1+x^2)^{-1}\in L^1(\R)\},\\
\mathcal H := \{ u\in \mathcal L\::\; (-\Delta)^\frac{1}{4}u \in L^2(\R)
\}.
\end{aligned}
\end{equation}
These spaces are endowed with the norms 
$$
\|u\|^2_{\mathcal{L}} = \int_{\R} \frac{u(x)^2}{1+|x|^2}dx, \quad  \text{ and } \quad \|u\|_{\mathcal H}^2 = \|(-\Delta)^\frac{1}{4}u\|_{L^2(\R)}^2+ \|u\|_{\mathcal{L}}^2. 
$$

{
It is known that one can construct an isometry between $L^2(\R)$ and $\mathcal L$ and between $H^{\frac{1}{2}}(S^1)$ and $\mathcal H$ via the standard sterographic projection. In particular, $\mathcal H $ is compactly embedded into $\mathcal L$.

\begin{lemma}\label{nondeg}
Let $\Phi \in \mathcal H$ be a solution to \eqref{LiouvilleScaled}. Then $\exists$ $\kappa_0,\kappa_1\in \R$ s.t. $\Phi = \kappa_0 Z_0 + \kappa_1 Z_1$, where $Z_0$ and $Z_1$ are the functions in \eqref{ZsuR}.
\end{lemma}
\begin{proof}
First of all, we observe that any solution $\ph$ to \eqref{LiouvilleScaled} is smooth. This follows by standard regularity results (see \cite[Theorem 13]{Tommaso}, the appendix in \cite{JMMX}, and \cite[Corollary 2.4 and 2.5]{RosSer})

Using the density of $C^\infty_c(\R)$ in $\mathcal H$ (which can be proved using the arguments of \cite[Lemma 11 and Lemma 12]{FSV}, since $\|(-\Delta)^\frac{1}{4}u\|_{L^2(\R)}$ is equivalent to the Gagliardo seminorm), we can find a sequence $\psi_{n}\in C_c^\infty(\R)$ such that $\psi_n\to 1$ in $\mathcal H$ (note that constant functions belong to $\mathcal H$). Then, for any $n$, have
$$
\int_{\R} (-\Delta)^\frac{1}{4} \ph (-\Delta)^\frac{1}{4} \psi_n \, dx = \int_{\R} \ph (-\Delta)^\frac{1}{2}\psi_n\, dx  = \int_{\R} f_\ph \psi_n \,dx, 
$$
and passing to the limit as $n\to \infty$, we get 
\begin{equation}\label{zeroavg}
\int_{\R} f_\ph \, dx = 0. 
\end{equation}
Let us now consider the functions  $$\Gamma (x,y) = \frac{1}{\pi}\log\left( \frac{1+|y|}{|x-y|}\right) \quad \text{ and } \quad  \Phi (x) := \int_{\R} \Gamma(x,y) f_{\ph}(y) dy.$$
Since $\ph \in \mathcal L \subseteq L_{\frac{1}{2}}(\R)$, according to \cite[Lemma 2.4]{HydStructure}, we have $ \ph =  \Phi + c$ for some $c\in \R$. Now, observing that  $\Gamma(\frac{1}{x},y) =    \frac{\log |x|}{\pi}+ \Gamma(x,\frac1y)$ for any $x,y\in \R\setminus \{0\}$ with $x\neq \frac{1}{y}$, we have that 
\begin{equation}\label{constdiff}
\ph\left(\frac{1}{x}\right) = \Phi\left(\frac{1}{x}\right) + c  = \frac{\log |x|}{\pi}\underbrace{\int_{\R} f_\ph (y) dy}_{=0 \text{ by } \eqref{zeroavg}} + \int_{\R} \Gamma\left(x,\frac{1}{y}\right) f_\ph (y) dy +c = 2\int_{\R}\Gamma(x,z) \frac{\ph(\frac{1}{z})}{1+z^2}dz+c, 
\end{equation}
for a.e. $x\in \R\setminus \{0\}$. Denoting $\tilde \ph (x):=  \ph(\frac{1}{x})$ and $f_{\tilde \ph}(x) :=2 \frac{\tilde \ph(x)}{1+x^2}$, via a simple change of variable we can show that $\tilde \ph \in \mathcal L$ and $f_{\tilde \ph}\in L^1(\R)$. Since $f_{\tilde \ph}\in L^1(\R)$,  \cite[Lemma 2.3]{HydStructure} implies that  
$$
\tilde \Phi(x): = \int_{\R} \Gamma(x,z)f_{\tilde \ph}(z)dz
$$
is a distributional solution to $(-\Delta)^\frac{1}{2}\tilde \Phi = f_{\tilde \ph} $ in $\R$. Moreover, using that $\tilde \ph \in \mathcal L$ (and in particular $f_{\tilde \ph} \in L^2_{loc}(\R)$) we can repeat the first part of the proof and show that $\tilde \Phi \in C^\infty(\R)$. By \eqref{constdiff},  we infer   that $\tilde \ph $ can be extended to a smooth function on $\R$. In particular this gives that $\ph \in L^\infty(\R)\cap C^\infty(\R)$. Then, we can conclude using directly the classification result in Theorem 1.4 of \cite{Santra}.
\end{proof}
}

%It is well known that the inclusion $\mathcal H \hookrightarrow \mathcal L$ is compact. 

In the following, for $\Vxi \in \mathcal P_{k,\eta}$ and $\lambda>0$,  we shall denote $Z_{i,j}(x):= Z_i(\frac{x-\xi_j}{\de_j})$, $i=0,1$, $j=1,\ldots,k$, where $\delta_j$ is defined as in \eqref{delta}.  Namely, we consider
\begin{equation}\label{Z}
Z_{0,j}(x):= \frac{\delta_j^2-(x-\xi_j)^2}{\delta_j^2+(x-\xi_j)^2} \quad \text{ and } \quad Z_{1,j}(x):= \frac{2\delta_j (x-\xi_j)}{\delta_j^2+(x-\xi_j)^2},
\end{equation}
which are solutions of the problem 
$$
\fl \ph = e^{U_{\de_j,\xi_j}} \ph \quad \text{ in }\R.
$$
We let $PZ_{i,j}:= (-\Delta)^{-\frac{1}{2}} \left( e^{U_{\de_{j},\xi_j}} Z_{i,j}\right)$  be the projection of $Z_{i,j}$ on $X^\frac{1}{2}_0(I)$. Then, we have the following expansions.

{
\begin{lemma}\label{PZ}
As $\lambda \to 0$, we have 
$$
P Z_{0,j} = Z_{0,j}+1 +O(\lambda^2),
$$
$$
P Z_{1,j}= Z_{1,j}+ 2\de_j \frac{\partial H}{\partial \xi}(\xi_j,\cdot)  + O(\lambda^3),
$$
uniformly in $\R$, for $j=1,\ldots,k$.  In particular $PZ_{0,j}=O(\lambda^2)$ and $PZ_{1,j}= O(\lambda)$ in $\R\setminus (\xi_j-\frac{\eta}{2},\xi_j+\frac{\eta}{2})$. 
\end{lemma}
\begin{proof}
First, note that for any $x\in \R$ the function $\xi\rightarrow H(\xi,x)$ belongs to $C^1(I)$, with derivative
$$
\frac{\partial H}{\partial \xi} (\xi, x) = 
\begin{cases}
-\frac{1}{\pi} \frac{x +\frac{\xi \sqrt{1-x^2}}{\sqrt{1-\xi^2}}}{1-x \xi + \sqrt{(1-\xi^2)(1-x^2)}} & \quad \text{ for } x\in I,  \\
\frac{1}{\pi} \frac{1}{\xi-x}  & \quad \text{ for } x \in \R \setminus I. 
\end{cases}
$$
We claim that $\frac{\partial H}{\partial \xi}(\xi,\cdot)$ is $\frac{1}{2}$-harmonic in $I$, {for any $\xi \in I$.} We prove that this is true in the sense of distributions. To show this, we observe that
$$
\int_{\R}  \frac{\partial H}{\partial \xi} (\xi,x) \; \fl \ph {(x)} \,  dx = 0 , \quad \forall \ph\in C^\infty_c(I). 
$$ 
Indeed, if we take $\psi\in C^\infty_c (-1,1)$, we have 
$$
\begin{aligned}
 \int_{\R} \psi(\xi) \int_{\R}  \frac{\partial H}{\partial \xi} (\xi,x) \, \fl \ph (x) \,  dx d\xi & =   \int_{\R}\fl \ph (x) \int_{\R}   \psi(\xi) \frac{\partial H}{\partial \xi} (\xi,x)  d\xi dx   \\
&   =-    \int_{\R}\fl \ph{(x)}  \int_{\R}   \psi'(\xi)  H(\xi,x)d\xi dx\\
&   =-    \int_{\R}  \psi'(\xi) \int_{\R}    \fl \ph{(x)}  \,   H(\xi,x)\,dx\,d\xi =0, 
\end{aligned}
$$
where the last equality follows from $(-\Delta)^\frac{1}{2} H(\xi,x)=0$. Since $\varphi$ and $\psi$ are arbitrary we have proved the claim.

Now, the statement can be proved as in Lemma \ref{PU}. Let us fix $1\le j\le k$. Since $\frac{\partial H}{\partial \xi}(\xi_j,\cdot)$ is $\frac{1}{2}$-harmonic in $I$,  
%(see Lemma \ref{Harmonic}),
 the definitions of $PZ_{0,j}$ and {$PZ_{1,j}$} imply that also the functions $v_{0,j} := PZ_{0,j} - Z_{0,j} - 1$ and $v_{1,j} := PZ_{1,j} - Z_{1,j} - 2\pi \delta  \frac{\partial H}{\partial \xi}(\xi_j,\cdot)$ are $\frac{1}{2}-$harmonic in $I$. Additionally, for $x \in \R\setminus I$, we have that 
$$
v_{0,j} (x) =- \frac{\delta_j^2-(x-\xi_j)^2}{\delta_j^2+(x-\xi_j)^2}-1 = O(\de_j^2)=O(\lambda^2),
$$
$$
v_{1,j}(x)  = -\frac{2\delta_j (x-\xi_j)}{\de_j^2+(x-\xi_j)^2} + \frac{2\delta_j}{(x-\xi_j)}=    O(\delta_j^3)=O(\lambda^3).
$$
Thus, we conclude via the maximum principle as in the proof of Lemma \ref{PU}.
\end{proof}
}

\begin{rem}\label{Orth}
For $i,j\in \{0,1\}$ and $h,l\in \{1,\ldots,k\}$, we have the orthogonality condition 
\[\begin{split}
\int_{\R} (-\Delta)^\frac{1}{4} P Z_{i,h} \cdot (-\Delta)^\frac{1}{4} P Z_{j,l} \, dx & = \int_{\R} e^{U_{\de_l,\xi_l}} Z_{i,h} P Z_{j,l} \, dx  = \pi \delta_{i,j}\delta_{h,l} + O(\lambda),
\end{split}
\]
where $\delta_{i,j}$ denotes the Kronecker delta symbol. Indeed, for $h\neq l$ we have $PZ_{j,l} =O(\lambda)$ in $\R\setminus (\xi_l- \frac{\eta}{2},\xi_l+\frac{\eta}{2})$ and $e^{U_{\de_h,\xi_h}} =O(\lambda)$ in $(\xi_l-\frac{\eta}{2},\xi_l+\frac{\eta}{2})$, while for $h =l$, we have 
$$
\int_{\R} e^{U_{\de_l,\xi_l}} Z_{i,h} P Z_{j,l} dx  = \int_{\R} e^{U_{\de_l,\xi_l}} Z_{i,h} Z_{j,l} dx + O(\lambda) =  \int_{\R} \frac{2 Z_{i}(y)Z_j(y)}{1+y^2} dy  + O(\lambda)= \pi \delta_{i,j} + O(\lambda). 
$$
\end{rem}

A standard procedure consists in inverting the operator ${L}$ on the orthogonal of the space generated by the functions $PZ_{i,j}$, $i=0,1$, $j=1,\ldots,k$, which can be considered as an approximate kernel for $ L$.   However,  Lemma \ref{PZ} shows that $PZ_{0,j}$ is not close to $Z_{0,j}$, as their difference approaches $1$ as $\lambda\to 0$.
%In other words, this means that $PZ_{0,j}$ cannot be considered an approximate solution of  ${\color{blue} L}\ph =0$.  
For this reason we can construct a smaller approximate kernel for $ L$ using only the functions $PZ_{1,j}$, $j=1,\ldots,k$.

In the following, for $\Vdelta = (\delta_1,\ldots,\delta_k)$ {defined as in \eqref{delta} and} for any $\Vxi \in \mathcal P_{k,\eta}$, we shall denote
{ $$K_{\Vde,\Vxi}= <\{P Z_{1,j}, \; j\in \{1,\ldots k\} \}>.$$ Let also $\pi$ and $\pi^\perp$  be the projections of $X^\frac{1}{2}_0(I)$ respectively into $K_{\Vde,\Vxi}$ and $K_{\Vde,\Vxi}^\perp$. 
%
%then, we can rewrite our equation 
%$$
%L \ph = E + N(\ph)
%$$ 
%as
%\begin{equation}\label{L_tilde_eq}
% L\ph =  (-\Delta)^{-\frac{1}{2}}  E +  (-\Delta)^{-\frac{1}{2}}  N(\ph). 
%\end{equation}
We now establish the invertibility of $L$ on $K_{\Vde,\Vxi}^\perp$. }
% Since we are working in a Hilbert space, and since $ L$ is  a Frehdolm operator of index $0$ (since it decomposes as the identity of $X^\frac{1}{2}_0(I)$ plus a compact operator),  it is enough to prove that $ L$ is injective to have the invertibility. 
%In the following, it will be convenient to consider the spaces 
%$$
%\begin{aligned}
%\mathcal L := \{u \in L^1_{loc}(\R)\;:\; |u|^2(1+x^2)^{-1}\in L^1(\R)\},\\
%\mathcal H := \{ u\in \mathcal L\::\; (-\Delta)^\frac{1}{4}u \in L^2(\R)
%\}.
%\end{aligned}
%$$
%It is well known that the inclusion $\mathcal H \hookrightarrow \mathcal L$ is compact. 

\begin{lemma}\label{bound_psi}
There exist $\bar \lambda$, $C>0$  such that 
\[
\| \psi\|\le C |\log \lambda| \| \pi^\perp L \,\psi\|.  
\]
for any $\lambda \in (0,\bar \lambda)$, $\Vxi \in \mathcal{P}_{k,\eta}$ and   $\psi \in K_{\Vdelta,\Vxi}^\perp$, with $\Vdelta$ given by \eqref{delta}.
\end{lemma}

\begin{proof}
We argue by contradiction. Suppose that there exist $\lambda_n \to 0$, $\Vxi_n =(\xi_{1,n},\ldots,\xi_{k,n}) \in \mathcal{P}_{k,\eta}$, and $\psi_n\in {K_{\Vdelta_n,\Vxi_n}}$ (where $\Vde_n =(\de_{1,n},\ldots,\de_{k,n})$ with $\de_{j,n}=\de_{i,n}(\lambda_n,\xi_n)$ given by \eqref{delta}) such that 
$$
\|\psi_n\|=1, \quad  \text{ and } \quad|\log \lambda_n|\|h_n\|\to 0, \quad \text{ where }\quad h_n:= \pi^\perp L\,\psi_n.
$$
Throughout this proof we will write $f_n:= f_{\lambda_n}, $ $\wn := \omega_{\Va,\Vde_n,\Vxi_n}$ and $U_{i,n}= U_{\de_{i,n},\xi_{i,n}}$. For any $i=1,\ldots,k$, we also let $Z_{0,i,n}$ and $Z_{1,i,n}$ denote the functions in \eqref{Z} with $\xi_i=\xi_{i,n}$ and $\delta_j= \delta_{i,n}$.

 By definition of ${\pi}^\perp$ there exists  $\zeta_n \in {K_{\Vdelta_n,\Vxi_n}}$ such that 
$ L \psi_n  = h_n + \zeta_n$. This means that 
\begin{equation}\label{testbis}
\begin{aligned}
\int_{\R} (-\Delta)^\frac{1}{4} \psi_n   (-\Delta)^\frac{1}{4} v  dx & = \int_{\R} f_n'(\wn) \psi_n v \, dx \\ &  +  \int_{\R} (-\Delta)^\frac{1}{4} h_n   (-\Delta)^\frac{1}{4} v  \, dx + \int_{\R} (-\Delta)^\frac{1}{4} \zeta_n   (-\Delta)^\frac{1}{4} v \, dx,
\end{aligned}
\end{equation}
for any $v\in X^\frac{1}{2}_0(I)$. Note that taking $v=\psi_n \in K_{\Vde_n,\Vxi_n}^\perp$, one finds
$$
\|\psi_n\|^2 = \int_{\R} f_n'(\wn) \psi_n^2 dx + \int_{\R} (-\Delta)^\frac{1}{4} \psi_n (-\Delta)^\frac{1}{4} h_n \,dx = \int_{\R} f_n'(\wn) \psi_n^2 dx  + O(\|h_n\|),
$$
from which we get 
\begin{equation}\label{convto1}
\int_{\R} f_n'(\wn) \psi_n^2 dx \to 1,
\end{equation}
as $n\to \infty$. Since $f_n'(\wn)$ is bounded in $L^1(I)$ by Lemma \ref{ExpDer}, Holder's inequality also gives 
\begin{equation}\label{bound}
\int_{\R} f_n'(\wn) |\psi_n|  dx  \le \left(\int_{\R} f_n'(\wn) \psi_n^2dx\right)^\frac{1}{2} \left(\int_{I} f_n'(\wn)dx\right)^\frac{1}{2} =O(1). 
\end{equation}
Keeping in mind the relations above, we split the rest of the proof into several steps.

\medskip\medskip
{\bf Step 1} Since $\zeta_n \in K_{\Vdelta_n,\Vxi_n}$, we can write $\zeta_n = \sum_{i=1}^k  c_{i,n} P Z_{1,i,n}$. \emph{We have $c_{i,n}=O(\|\zeta_n\|)$ for $i=1,\ldots,k$. In particular $\|\zeta_n\|_{L^\infty(\R)} = O(\|\zeta_n\|)$. }

\medskip
\noindent Indeed, setting $\bar c_n= \max\{|c_{i,n}|\;:\; 0\le i\le k\}$ and using Remark \ref{Orth}, we find that 
$$
\begin{aligned}
\|\zeta_n\|^2  & = \sum_{i=1}^k\sum_{j=1}^k  c_{i,n} c_{j,n}  \int_{\R} (-\Delta)^\frac{1}{4} PZ_{1,i,n}  (-\Delta)^\frac{1}{4} PZ_{1,j,n} dx \\
%& = \sum_{i=1}^k\sum_{j=1}^k c_{i,n} %c_{j,n}  \int_{\R} e^{U_{\de_{i,n},
%\xi_{i,n}}} Z_{1,\delta_{i,n},\xi_{i,n}}  %PZ_{1,\delta_{j,n},\xi_{j,n}} \\
%& = \sum_{i=1}^k  c_{i,n}^2  \int_{\R} %e^{U_{\de_{i,n},\xi_{i,n}}} Z_{1,\delta_{i,n} ,\xi_{i,n}} PZ_{1,\delta_{i,n} ,\xi_{i,n}}  +  O(\lambda_n \bar c_n^2 )\\
 & = \pi \sum_{i=1}^k  c_{i,n}^2  +  O(\lambda_n \bar c_n^2 )\\
&  \ge \pi \bar c_n^2 + O(\lambda_n \bar c_n^2). 
\end{aligned}
$$
Since $\lambda_n \to 0$, this implies that $\bar c_n = O(\|\zeta_n\|)$.

\medskip{
{\bf Step 2}: \emph{For $i=1,\ldots,k$, and $s= 1,2$, we have that 
\begin{equation}\label{bound1}
\int_{\xi_{i,n}-\frac{\eta}{2}}^{\xi_{i,n}+\frac{\eta}{2}} |f_n'(\wn) - e^{U_{i,n}}| |\psi_n|^s \,dx = O(\sqrt{\lambda_n}). 
\end{equation}
Moreover, we have 
\begin{equation}\label{bound2}
\int_{\R} e^{U_{i,n}} |\psi_n|^s dx =O(1) \quad \text{ and } \int_{\R} |f_n'(\wn) - e^{U_{i,n}}| |\psi_n|^s  |PZ_{j,i,n}|\,dx =O(\sqrt{\lambda_n}).
\end{equation}
}
Indeed, in view of Lemma \ref{ExpDer}, we have  
\begin{equation}\label{trivialbound}
\int_{\xi_{i,n}-\frac{\eta}{2}}^{\xi_{i,n}+\frac{\eta}{2}} |f_n'(\wn) - e^{U_{i,n}}| |\psi_n|^s dx = \int_{\xi_{i,n}-\frac{\eta}{2}}^{\xi_{i,n}+\frac{\eta}{2}} e^{U_{i,n}}|\psi_n|^sO(|x-\xi_{i,n}|)dx +  \int_{\xi_{i,n}-\frac{\eta}{2}}^{\xi_{i,n}+\frac{\eta}{2}} e^{U_{i,n}}|\psi_n|^s O(\lambda_n^2) dx .  
\end{equation}
By Holder's inequality, estimate \eqref{gen Est} and Sobolev's inequality, we get
\begin{equation}\label{radice}
\int_{\xi_{i,n}-\frac{\eta}{2}}^{\xi_{i,n}+\frac{\eta}{2}} e^{U_{i,n}} |x-\xi_{i,n}| |\psi_n|^s dx \le \left(\int_{\xi_{i,n}-\frac{\eta}{2}}^{\xi_{i,n}+\frac{\eta}{2}}e^{2U_{i,n}}|x-\xi_{i,n}|^2dx \right)^\frac{1}{2}\|\psi_n^s\|_{L^2(I)} = O(\sqrt{\lambda_n}).  
\end{equation}
Furthermore, using again Lemma \ref{ExpDer} and \eqref{convto1}-\eqref{bound}, we find that
$$
\begin{aligned}\int_{\xi_{i,n}-\frac{\eta}{2}}^{\xi_{i,n}+\frac{\eta}{2}} e^{U_{i,n}}  |\psi_n|^s dx &= \int_{\xi_{i,n}-\frac{\eta}{2}}^{\xi_{i,n}+\frac{\eta}{2}} f_n'(\omega_n)  |\psi_n|^s dx  + O(\sqrt{\lambda_n} ) + O\left(\lambda_n^2 \int_{\xi_{i,n}-\frac{\eta}{2}}^{\xi_{i,n}+\frac{\eta}{2}} e^{U_{i,n}}  |\psi_n|^s dx \right)\\
& = O(1)+   O\left(\lambda_n^2 \int_{\xi_{i,n}-\frac{\eta}{2}}^{\xi_{i,n}+\frac{\eta}{2}} e^{U_{i,n}}  |\psi_n|^s dx \right),
\end{aligned}
$$
which implies that 
\begin{equation}\label{FinBound}
\int_{\xi_{i,n}-\frac{\eta}{2}}^{\xi_{i,n}+\frac{\eta}{2}} e^{U_{i,n}}  |\psi_n|^s dx  =O(1).
\end{equation}
Then, we get \eqref{bound1} by substituting \eqref{radice} and \eqref{FinBound}  in \eqref{trivialbound}. The first estimate in \eqref{bound2} follows by \eqref{FinBound} and the bound $e^{U_{i,n}} =O(\lambda_n)$ in $\R\setminus (\xi_{i,n}-\frac{\eta}{2},\xi_{i,n}+\frac{\eta}{2})$. Similarly, the second estimate in \eqref{bound2} is a consequence of \eqref{bound1} and the bounds $PZ_{0,i,n} = O(\lambda_n^2)$, $PZ_{1,i,n}=O(\lambda_n)$ in $\R\setminus (\xi_{i,n}-\frac{\eta}{2},\xi_{i,n}+\frac{\eta}{2})$.}

\medskip\medskip
{\bf Step 3}: \emph{We have $\|\zeta_n\| = {o(|\log\lambda_n|^{-1})}$ as $n\to \infty$.} 
\medskip

Taking $v = \zeta_n$ in \eqref{testbis} and recalling that $\zeta_n \in K_{\Vdelta_n,\Vxi_n}$, $\psi_n, h \in K_{\Vdelta_n,\Vxi_n}^\perp$ we find that 
\begin{equation}\label{testzetabis}
\begin{aligned}
0 & = \int_{\R} f_n'(\wn) \psi_n \zeta_n dx 
% +  \int_{\R} (-\Delta)^\frac{1}{4} h_n   (-\Delta)^\frac{1}{4} \zeta_n  dx 
 + \|\zeta_n\|^2 \\
% & = \int_{\R} f_n'(\wn) \psi_n \zeta_n dx  
 %+ O(\|h_n\|\|\zeta_n\|) 
 %+ \|\zeta_n\|^2 \\
 & =   \sum_{i=1}^k  c_{i,n} \int_{\R} f_n'(\wn)  \psi_n P Z_{1,i,n} dx
% +   O(\|h_n\|\|\zeta_n\|) 
 + \|\zeta_n\|^2.
\end{aligned}
\end{equation}

Now, for $i=1,\ldots ,k $ and Step 2 and Lemma \ref{PZ} give 
\[
\int_{\R} f_n'(\wn)  \psi_n P Z_{1,i,n}dx    = \int_{\R}  e^{U_{i,n}}PZ_{1,i,n} \psi_n dx + O(\sqrt{\lambda_n})  = \underbrace{\int_{\R}   \psi_n  (-\Delta)^\frac{1}{2}PZ_{1,i,n} dx}_{=0 \text{ by }\psi_n\in K^{\perp}_{\Vde_n,\Vxi_n}}  +O(\sqrt{\lambda_n}).
\]
 Then, using also Step 1, we can rewrite \eqref{testzetabis} as 
$$\|\zeta_n\| = O(\sqrt{\lambda_n})= o(|\log \lambda_n|^{-1}).$$
%$$
%\|\zeta_n\|^2 = O(\sqrt{\lambda_n}\|\xi_n\|)+ O(\|h_n\|\|\xi_n\|),
%$$
%which gives $\|\zeta_n\| = O(\sqrt{\lambda_n})+O(\|h_n\|) = o(|\log \lambda_n|^{-1})$, 
%since  $\|h_n\||\log \lambda_n|\to 0$.

\medskip\medskip
{{\bf Step 4} \emph{For $i=1,\ldots,k$, we have that  
$$
\int_{\R} e^{U_{i,n}} \psi_n \, dx = o(|\log \lambda_n|^{-1}) \quad \text{ and }\quad 
\int_{\R} e^{U_{i,n}} U_{i,n} \psi_n \, dx \to 0.
$$
}}
First of all, taking $v=PZ_{0,i,n}$ in \eqref{testbis}, { and  using Steps 2-3 and {Lemma \ref{PZ}}, we find that 
$$\begin{aligned}
\int_{\R} (-\Delta)^\frac{1}{4} PZ_{0,i,n} \cdot (-\Delta)^\frac{1}{4}\psi_n dx & = \int_{\R} f_n'(\wn) \psi_{n} PZ_{0,i,n} dx  + O(\|PZ_{0,i,n}\|\|h_n\|)  +O(\|PZ_{0,i,n}\|\|\zeta_n\|)\\
 & = \int_{\R} e^{U_{i,n}} \psi_n PZ_{0,i,n} dx + o(|\log \lambda_n|^{-1})
 \\ &=     \int_{\R} e^{U_{i,n}} \psi_n Z_{0,i,n}dx + \int_{\R} e^{U_{i,n}} \psi_n dx + o(|\log \lambda_n|^{-1}). 
\end{aligned}
$$}
Besides, by definition of $PZ_{0,i,n}$, we have 
$$
\int_{\R} (-\Delta)^\frac{1}{4} PZ_{0,i,n} \cdot (-\Delta)^\frac{1}{4}\psi_n dx  = \int_{\R} e^{U_{i,n}} \psi_n Z_{0,i,n}\, dx .
$$
If we combine the two estimates above,  we find that 
\[
\int_{\R} e^{U_{i,n}} \psi_n\, dx = o(|\log\lambda_n|^{-1}).
\]
Since $\|U_{i,n}\|_{L^\infty(I)}=O(|\log \lambda_n|)$ (in fact $0\le |x-\xi_i|\le 2$ implies $\frac{2\de_{i,n}}{\delta_{i,n}^2+4}\le e^{U_{i,n}}\le \frac{2}{\delta_i}$ in $I$) and $\psi_{n} =0$ in $\R \setminus I$, we get the conclusion.

%In view of Lemma \ref{ExpDer} and Step 2, we also get 
%\begin{equation}\label{conv2}
%\int_{\R} f_n'(\wn)\psi_n dx  = o(|\log\lambda_n|^{-1}). 
%\end{equation}
%Next, taking $v= P U_{i,n}$ in \eqref{testbis}, and using the definition of $P U_{i,n}$, we get 
%\begin{equation}\label{crucial}
%\int_{\R} e^{U_{i,n}}\psi_n dx  = \int_{\R} f_n'(\wn) \psi_n P U_{i,n} dx + O(\|P U_{i,n}\|\|h_n\|) + o(1),
%% O(|\xi_n|\lambda_n)
%\end{equation}
%and, thanks to Lemma \ref{PU} and Lemma \ref{ExpDer}, we find that 
%$$\begin{aligned}
%\int_{\R} f_n'(\wn) \psi_n P U_{i,n} dx & =  \underbrace{-\log (2\delta_{i,n}) \int_{\R} f_n'(\wn) \psi_n\, dx}_{{\to 0 \text{ by } \eqref{conv2}}}  + \underbrace{\int_{\R} f_n'(\wn) U_{i,n} \psi_n \, dx}_{=:I_n} \\
%& \quad + 2\pi \int_{\R} f_n'(\wn) H(\xi_{i,n},x) \psi_n \, dx +  O(\de_{i,n}^2)\\
%%& =  I_n + o(1) + \sum_{j=1}^k H(\xi_{i,n},\xi_{j,n}) \int_{\R} f_n'(\wn) \psi_n dx +   \sum_{j=1}^k \int_{\xi_{j,n} - \frac{\eta}{2}}^{\xi_{j,n}+\frac{\eta}{2}}  f_n'(\omega_{a,\delta_{i,n} , \xi_{i,n}})\psi_n  O(|x-\xi_{j,n}|) dx\\
%& = I_n + o(1) +   \sum_{j=1}^k \int_{\xi_{j,n} - \frac{\eta}{2}}^{\xi_{j,n}+\frac{\eta}{2}}  O(e^{U_{j,n}}|x-\xi_{j,n}| |\psi_n|) dx \\
%& = I_n +o(1),
%\end{aligned}
%$$
%where the last inequality relies on \eqref{radice}. Therefore, using \eqref{conv1} and \eqref{crucial}, we obtain  
%$$
%I_n = O(\|PU_{i,n}\|\|h_n\|) + o(1). 
%$$
%Since 
%$$
%\|P U_{i,n}\|^2 = \int_{\R} e^{U_{i,n}} PU_{i,n} dx  = O(|\log \lambda_n|),
%$$
%and $\|h\| =o(|\log \lambda_n |^{-1})$, we can conclude that 
%$$
%I_n \to 0,
%$$
%as claimed. 

\medskip
{\bf Step 5}: \emph{For $i=1,\ldots, k$, the function $\Psi_{i,n} := \psi_n(\xi_{i,n}+ \de_{i,n} \cdot)$ satisfies $\Psi_n \to 0$ in $\mathcal{L}$, where $\mathcal L$  is defined in \eqref{spaces}. }

First of all, we observe that 
$$
\|\Psi_{i,n}\| = \|\psi_n\| = 1 \quad \text{ and }\quad 
2 \int_{\R} \frac{|\Psi_{i,n}|^2}{1+|x|^2} dx   = \int_{\R} e^{U_{{i,n}}} \psi_n^2 dx   \le C,
$$
by Step 2. Then, $\Psi_{i,n}$ is uniformly bounded in the space 
$\mathcal H 
$ (see \eqref{spaces}), 
which is compactly embedded in $\mathcal{L}.$ Thus, there exists $\Psi_{\infty} \in \mathcal{H}$  such that, up to subsequences, we have $\Psi_{i,n}\rightharpoonup\Psi_{\infty}$ weakly in $\mathcal H$ and$\Psi_{i,n}\to \Psi_{\infty}$ in $\mathcal L$ as $n\to +\infty$. The weak convergence in $\mathcal{H}$ implies that 
$$
\int_{\R} (-\Delta)^\frac{1}{4} \Psi_{i,n} \cdot (-\Delta)^\frac{1}{4} w \; dx \to  \int_{\R} (-\Delta)^\frac{1}{4} \Psi_{\infty} \cdot (-\Delta)^\frac{1}{4} w \, dx,
$$ 
for any $w \in C^\infty_c(\R)$. Besides, using \eqref{testbis} with $v=v_n := w( \frac{\cdot-\xi_{i,n}}{\de_{i,n}})$, we get
$$\begin{aligned}
\int_{\R} (-\Delta)^\frac{1}{4} \Psi_{i,n} \cdot (-\Delta)^\frac{1}{4} w \; dx & = \int_{\R} (-\Delta)^\frac{1}{4} \psi_{n} \cdot (-\Delta)^\frac{1}{4} v_n  \; dx \\
& = \int_{\R}   f_n'(\wn) \psi_n  v_n \, dx + \int_{\R} (-\Delta)^{\frac{1}{4}} \tilde h_n  (-\Delta)^{\frac{1}{4}} v_n dx,
\end{aligned} 
$$
where $\tilde h_n = h_n + \zeta_n$. Since $\tilde h_n \to 0$ (by Step 3), we get that 
$$
\int_{\R} (-\Delta)^{\frac{1}{4}} \tilde h_n  (-\Delta)^{\frac{1}{4}} v_n dx \le \|\tilde h_n\| \|v_n\| = \|\tilde h_n\|\|w\| \to 0. 
$$
Moreover, noting that $v_n$ is supported in $(\xi_{i,n}-R \de_{i,n},\xi_{i,n}+\de_{i,n} R)$ for some $R>0$, we have 
$$\begin{aligned}
\int_{\R}   f_n'(\wn) \psi_n v_n dx  & = \int_{\xi_{i,n}-R \de_{i,n}}^{\xi_{i,n}+R\de_{i,n}} e^{U_{i,n}} (1+O(|x-\xi_{i,n}|)+O(\lambda_n)) \psi_n v_n \\
& = \int_{-R}^R \frac{2}{1+y^2} ( 1+ O(\lambda_n(|y|+1))) \Psi_{i,n} w \, dy \\
& \to \int_{\R} \frac{2}{1+y^2} \Psi_\infty w \, dy,
\end{aligned}
$$
where the convergence in the last line follows by the convergence of $\Psi_{i,n}$ in $\mathcal{L}$. 
Then, it follows that $\Psi_{\infty}$ is a solution in $\mathcal{H}$ to the problem 
$$
(-\Delta)^\frac{1}{2} \Psi_{\infty} = \frac{2}{1+x^2} \Psi_{\infty} \quad \text { in } \R.
$$
Then, by Lemma \ref{nondeg}, there exist $\kappa_0,\kappa_1 \in \R$ such that  $\Psi_{\infty} = \kappa_0 Z_0 + \kappa_1 Z_1.$ But using again the convergence in $\mathcal L$ and recalling that $\psi_n \in K_{\Vde_{i,n},\Vxi_{i,n}}^\perp$, we have 
$$\begin{aligned}
0 = \int_{\R} (-\Delta)^\frac{1}{4} \psi_n (-\Delta)^\frac{1}{4} P Z_{1,i,n}\, dx & =  \int_{\R}  \psi_n e^{U_{i,n}} Z_{1,i,n} \, dx \\
& = \int_{\R}  \frac{2\Psi_{i,n} Z_1}{1+y^2} dy \\
%& \to  \int_{\R} \frac{2\Psi_{\infty} Z_1}{1+y^2} dy \\
& \to  \kappa_0\int_{\R} \frac{2Z_0 Z_1}{1+y^2} dy + \kappa_1\int_{\R} \frac{2Z_1^2}{1+y^2} dy\\
& = \pi \kappa_1,
\end{aligned}$$
Hence, $\kappa_1=0$.  Similarly, thanks to Step 4,  we know that 
$$\begin{aligned}
0 
%= \lim_{n\to \infty} \int_{\R}  f_n'(\wn) U_{\delta_{i,n}, \xi_{i,n}} \psi_n & = \lim_{n\to \infty}  \int_{\R}  \psi_n e^{U_{\de_{i,n},\xi_{i,n}}} U_{{\de_{i,n},\xi_{i,n}}} dx \\
%& 
&= \lim_{n\to \infty}  \int_{\R}  \psi_n e^{U_{\de_{i,n},\xi_{i,n}}}  ( U_{{\de_{i,n} ,\xi_{i,n}}} +\log \de_{i,n}) dx \\
& =\lim_{n\to \infty} \int_{\R}  \frac{2\Psi_{i,n}}{1+y^2} \log \left( \frac{2}{1+y^2} \right) dy \\
%& \to \int_{\R}  \frac{2\Psi_{\infty}}{1+y^2} \log \left( \frac{2}{1+y^2} \right) dy \\
& =  \kappa_0\int_{\R}  \frac{2Z_0}{1+y^2} \log \left( \frac{2}{1+y^2} \right) dy \\
& = \pi \kappa_0,
\end{aligned}$$
which implies $\kappa_0=0$ and $\Psi_\infty \equiv 0$.

\medskip
{\bf Step 6}: Conclusion of the proof.
We know that
$$
1+o(1) = \int_{\R} f_n'(\wn) \psi_n^2 dx  
 = \sum_{i=1}^k \int_{\xi_{i,n}-\frac{\eta}{2}}^{\xi_{i,n}+\frac{\eta}{2}} e^{U_{i,n}} \psi_n^2\,dx + O(\sqrt{\lambda_n})$$
by \eqref{convto1}, Lemma \ref{ExpDer} and Step 2. But, using step 5, one gets
$$\begin{aligned}
\int_{\xi_{i,n}-\frac{\eta}{2}}^{\xi_{i,n}+\frac{\eta}{2}} e^{U_{i,n}} \psi_n^2\,dx\le \int_{\R} e^{U_{i,n}} \psi_n^2 dx =    \int_{\R} \frac{2\Psi_{i,n}^2}{1+y^2} dy   \to 0,
\end{aligned} $$
for any $i=1,\ldots,k$. This gives  a contradiction. 
\end{proof}

The a-priori estimates of Lemma \ref{bound_psi} imply the following invertibility property. 

\begin{cor}\label{lemmaA}
For $\lambda\in (0,\bar \lambda)$ and $\Vxi \in \mathcal P_{k,\eta}$, the operator $A:=\pi^{\perp} L:K_{\Vde,\Vxi}^\perp\rightarrow K_{\Vde,\Vxi}^\perp$ is invertible and 
$$\|A^{-1}\|_{\mathcal{L}(K_{\Vde,\Vxi}^\perp)}=O(|\log \lambda|),$$
where $\|F\|_{\mathcal{L}(K_{\Vde,\Vxi}^\perp)}:=\sup_{\{h\in K_{\Vde,\Vxi}^\perp,\ \|h\|\leq 1\}}\|Fh\|.$
\end{cor}
\begin{proof} 
By Lemma \ref{bound_psi}, for any $\ph\in K_{\Vde,\Vxi}^\perp$, %setting $h:=A\ph$, 
we have
$$
\| \ph\|\leq C |\log\lambda| \|A\ph\|.$$ In particular, $A$ is injective. Since $K_{\Vde,\Vxi}^\perp$ is a Hilbert space, and since $A$ is  a Fredholm operator of index $0$ (indeed it decomposes as the identity of $K_{\Vde,\Vxi}^\perp$ plus a compact operator), we can assert that $A$ is invertible. Moreover, we have 
\begin{equation*}
\begin{split}
\|A^{-1}\|_{\mathcal{L}(K_{\Vde,\Vxi}^\perp)}&=\sup_{\{h\in K_{\Vde,\Vxi}^\perp,\ \|h\|\leq 1\}}\|A^{-1}h\|
%=\sup_{\{h\in K_{\Vde,\Vxi}^\perp,\ \|h\|\leq 1\}}\|\ph\|\\
\leq C |\log \lambda|.
\end{split}
\end{equation*}
\end{proof}

\section{Fix point argument}\label{Sec:Fix}

%Let ${L}$ be the operator defined by \eqref{DefTildeL}. After projecting over the Kernel  $K_{\de,\xi}= <\{P Z_{1,j}, j\in \{1,\ldots k\} \}>$, equation \eqref{proj1} provides us with the two following equations
%\begin{equation}\label{proj1_fix}
%\pi^{\perp} L\ph =  \pi^{\perp}(-\Delta)^{-\frac{1}{2}}  E +  \pi^{\perp}(-\Delta)^{-\frac{1}{2}}  N(\ph). 
%\end{equation}
%
%\begin{equation}\label{proj2_fix}
%\pi L\ph =  \pi(-\Delta)^{-\frac{1}{2}}  E + \pi (-\Delta)^{-\frac{1}{2}}  N(\ph). 
%\end{equation}

As we have outlined in Section \ref{Sec:Prelim}, equation \eqref{Eq} can be reduced to the  couple of nonlinear problems \eqref{proj1}-\eqref{proj2}. With the notation of the previous section, let us consider the operator $A=\pi^\perp  L: K_{\Vde,\Vxi}^\perp \rightarrow K_{\Vde,\Vxi}^\perp$. Thanks to Corollary \ref{lemmaA}, we can rewrite equation \eqref{proj1} as 

$$
\ph =A^{-1}\left(\pi^{\perp}(-\Delta)^{-\frac{1}{2}}  E +  \pi^{\perp}(-\Delta)^{-\frac{1}{2}}  N(\ph)\right). 
$$

We now prove that this equation admits a solution for any small $\lambda$ and any  $\Vxi \in \mathcal{P}_{k,\eta}$.

\begin{lemma}\label{LemmaFix}
Let $p\in (1,2)$ be fixed. Then, there exist {$\kappa=\kappa(p,\eta)>0$ and $\lambda_0 = \lambda_0(p,\eta)>0$} such that, for any $\Vxi \in\mathcal P_{k,\eta}$ and $\lambda \in (0,\lambda_0)$,  the operator $$T(\ph):=A^{-1}\left(\pi^{\perp}(-\Delta)^{-\frac{1}{2}}  E +  \pi^{\perp}(-\Delta)^{-\frac{1}{2}}  N(\ph)\right)$$ has a fixed point on 
\begin{equation}\label{B}
B:=\{\ph \in K_{\Vde,\Vxi}^\perp\;:\;\ \|\ph\|\leq 
%2C_AC(p)C_E.
\kappa|\log \lambda|\lambda^{1/p} \}.\end{equation}
\end{lemma}

{\begin{proof}
By Corollary \ref{lemmaA}, Lemma \ref{Est_E}, estimate \eqref{EllEst}, and the fact that the projection $\pi^\perp$ reduces the norm, we can find constants $C_A$ and $C_E$ such that  
\begin{equation}\label{CACE}
\|A^{-1}\|_{\mathcal{L}(K_{\Vde,\Vxi}^\perp)}\leq C_A|\log \lambda| \quad \text{ and } \quad  \|\pi^\perp (-\Delta)^\frac{1}{2}E\| \le C_E \lambda^\frac{1}{p}. 
\end{equation}
for any $\Vxi \in \mathcal P_{k,\eta}$ and any small $\lambda$.
Note that $C_A$ and $C_E$ do not depend neither on  $\Vxi$ nor on $\lambda$.  Similarly, for any $s>p$, Lemma \ref{est_L} and estimate \eqref{EllEst} imply the existence of a constant $C_N$, depending only on $s,p,\eta$ such that 
\begin{equation}\label{CN}
\|\pi^\perp(-\Delta)^\frac{1}{2} (N(\ph_1)-N(\ph_2))\|\le C_N \lambda^{\frac{1}{s}-1} |\log \lambda| \|\ph_1-\ph_2\|(\|\ph_1\| + \|\ph_2\|),
\end{equation}
for any $\Vxi \in \mathcal P_{k,\eta}$, $\lambda$ small enough and any $\ph_1,\ph_2\in X^\frac{1}{2}_0(I)$ with $\|\ph_i\|\le 1$ for $i=1,2$. 

Let us set $\kappa:=2C_AC_E$. We shall prove that $T$ is a contraction on $B$. First, taking $\lambda$ small enough so that $\kappa \lambda^\frac{1}{p}|\log \lambda|\le 1$,  we get $\|\ph\|\le 1$ for any $\ph \in B$. Hence, \eqref{CACE}  and \eqref{CN} give
\[
\begin{split}
\|T(\ph)\|
& \leq 
\|A^{-1}\|_{\mathcal{L}(K_{\Vde,\Vxi}^\perp)}\|\pi^{\perp}(-\Delta)^{-\frac{1}{2}}  E\| + \|A^{-1}\|_{\mathcal{L}(K_{\Vde,\Vxi}^\perp)} \|\pi^{\perp}(-\Delta)^{-\frac{1}{2}}  N(\ph)\|\\
&\le C_A C_E \lambda^\frac{1}{p}|\log \lambda| + C_A C_N \lambda^{\frac{1}{s}-1}\|\ph\|^2\\
& \le C_A C_E \lambda^\frac{1}{p}|\log \lambda| \left( 1 +4C_A^2C_E C_N  \lambda^{\frac{1}{p}+\frac{1}{s}-1}|\log \lambda| \right),
\end{split}
\]
where last inequality follows from the definition of $B$ in \eqref{B} and our choice of $\kappa$.  Since $p\in (1,2)$, it is enough to  take $s\in(p,\frac{p}{p-1})$ and $\lambda$  such that $4C_A^2C_EC_N |\log\lambda|^2\lambda^{\frac{1}{p}+\frac{1}{s}-1}\le 1$ to get $T(\ph)\in B$, $\forall\, \ph\in B.$ 

Arguing as above we now prove that $\|T\ph_1-T\ph_2\|\leq \frac{1}{2} \|\ph_1-\ph_2\|$, for all $\ph_1,\ph_2\in B$. Indeed, thanks to \eqref{CN}, it is sufficient to choose $\lambda$ small enough such that $$4C_A^2C_E C_N|\log \lambda|^2\lambda^{\frac{1}{p}+\frac{1}{s}-1}\leq \frac{1}{2},$$
to get
$$\begin{aligned}
\|T(\ph_1)-T(\ph_2)\|& = \|A^{-1}\pi^{\perp}(-\Delta)^{-\frac{1}{2}}  \left(N(\ph_1)-N(\ph_2)\right)\|
\\ &\le  C_A|\log\lambda| C_N\lambda^{\frac{1}{s}-1} \|\ph_1-\ph_2\| (\|\ph_1\|+\|\ph_2\|) \\
& \le \underbrace{4C_A^2C_E C_N|\log \lambda|^2\lambda^{\frac{1}{p}+\frac{1}{s}-1}}_{\le \frac{1}{2}} \|\ph_1-\ph_2\|.
\end{aligned}
$$ 
Thus we have proved that $T$ is a contraction on the ball $B$, so it has a unique fix point in $B$.
\end{proof} }

For $\Vxi\in \mathcal P_{k,\eta}$ and $\lambda$ small enough, let $\ph_{\lambda,\xi}$ be the fix point for the operator $T$ constructed in Lemma \ref{LemmaFix}. By definition,  $\ph_{\lambda,\xi}$ satisfies \eqref{proj1}. Then, since $K_{\Vde,\Vxi}$ is spanned by $PZ_{1,1},\ldots,PZ_{1,k}$, 
%the fixed point $\ph$ constructed in Lemma \ref{LemmaFix} is a weak solution to 
%\begin{equation}\label{EqPhi}
%\begin{aligned}
%(-\Delta)^{\frac{1}{2}} \ph_{\xi}-\f'(\omega_{a,\delta,\xi})\ph_{\xi}& =E+N(\ph_{\xi})+\sum_{i=1}^k(-\Delta)^{\frac{1}{2}}c_iPZ_{1,i} \\ & =E+N(\ph_{\xi})+\sum_{i=1}^k c_i  e^{U_{\de_i,\xi_i}}Z_{1,i} , 
%\end{aligned}
%\end{equation}
%More precisely,
as a consequence of Lemma \ref{LemmaFix}, we get the following proposition:

\begin{prop}\label{PropPhi}
Fix $p\in (0,1)$ and let $\lambda_0$ and $\kappa$ be as in Lemma \ref{LemmaFix}. Then, for any $\lambda\in (0,\lambda_0)$ and any $\Vxi \in \mathcal{P}_{k,\eta}$, there exists a unique function $\ph_{\lambda,\Vxi}\in K_{\de,\Vxi}^\perp$ such that $\|\ph_{\lambda,\Vxi}\|\le \kappa \lambda^\frac{1}{p}|\log \lambda|$ and such that we can find $k$ coefficients $c_i = c_i(\lambda,\Vxi),$ $i=1,...,k$, such that 
\begin{equation}\label{EqPhi}
\begin{aligned}
(-\Delta)^{\frac{1}{2}} \ph_{\lambda,\xi}-\f'(\omega_{a,\delta,\xi})\ph_{\lambda,\xi}& =E+N(\ph_{\lambda,\xi})+\sum_{i=1}^k c_i  e^{U_{\de_i,\xi_i}}Z_{1,i}.
\end{aligned}
\end{equation}
Moreover, by definition of $E$, $N$ and $ L$, we also have 
\begin{equation}\label{EqSum}
(-\Delta)^\frac{1}{2} (\w + \ph_{\xi,\Vde}) =  \f (\w + \ph_{\xi,\Vde}) +  \sum_{i=1}^k c_i  e^{U_{\de_i,\xi_i}}Z_{1,i}.
\end{equation}
\end{prop}

\begin{rem} By testing the equation \eqref{EqPhi} against $PZ_{1,i}$, $i=1,\ldots,k$, we get that 
\[
c_i(\lambda,\Vxi) = -\sum_{j=1}^k b^{ij} \int_{\R} \left( \f'(\w)\ph_{\lambda,\xi}+E +N(\ph_{\lambda,\xi})\right)PZ_{1,i}\, dx ,
\]
where the $b^{ij}=b^{ij}(\lambda,\xi)$ are the coefficients of the inverse of the matrix $(b_{ij})_{1\le i,j\le k}$ with 
$$
b_{ij} = \int_{\R} e^{U_{\de_i,\xi_i}}Z_{1,i}PZ_{1,j} \,dx .
$$
The matrix $(b_{ij})_{1\le i,j\le k}$ is symmetric and invertible by Remark \ref{Orth}.  
\end{rem}

We conclude this section by proving the regularity of $\pl$ with respect to $\Vxi$. From now on, with some abuse of notation we will use the notation $\lambda_0$ to refer different constants possibly smaller than the one given by Lemma \ref{LemmaFix} and Proposition \ref{PropPhi}.

\begin{lemma}
For any $\lambda \in (0,\lambda_0)$, the map $\Vxi\to \ph_{\lambda,\Vxi}$ is a $C^1$ map from $\mathcal P_{k,\eta}$ into $X_0^\frac{1}{2}(I)$. 
\end{lemma}
\begin{proof}
For the study of the regularity of $\ph_{\lambda,\Vxi}$ it is important to recall that $\pi$, $\pi^\perp$, $ L$ and $T$ depend on $\lambda$ and $\Vxi$. For this reason, throughout this proof these operators will be denoted respectively by $\pi_{\lambda,\Vxi},\pi_{\lambda,\Vxi}^\perp$, $ L_{\lambda,\Vxi}$ and $T_{\lambda,\Vxi}$. 
For a fixed $\lambda\in (0,\lambda_0)$, let us consider the $C^1$ map $G_\lambda : \mathcal P_{k,\eta}\times X_0^\frac{1}{2}(I)\ra X_0^\frac{1}{2}(I)$ defined by 
 $$G_{\lambda} (\xi, \ph) =\ph +  \pi_{\lambda,\Vxi}^\perp \left[\w  -  (-\Delta)^{-\frac{1}{2}} \f(\w+\pi_{\lambda,\Vxi}^\perp \ph)\right].$$ Note that
$$
\begin{aligned}
\frac{\partial G_{\lambda}}{\partial \ph} (\Vxi,\ph)[v] &  = v -\pi^\perp_{\lambda,\Vxi} (-\Delta)^{-\frac{1}{2}} f_\lambda'(\w +\pi_{\lambda,\Vxi}^\perp \ph)\pi_{\lambda,\Vxi}^\perp v,  
\end{aligned}
$$
for any $v\in X^\frac{1}{2}_0(I)$. In particular, $\frac{\partial G_{\lambda}}{\partial \ph} (\Vxi,\ph)$ is a Fredholm operator of index 0 and thus, it is invertible if and only if it is injective. 
By definition, we have that  $G_{\lambda}(\xi,\ph)=0$ if and only if  $\ph \in K^\perp_{\Vde,\Vxi}$ is a fix point for  $T_{\lambda,\Vxi}$.  In particular $G_\lambda(\Vxi,\ph_{\lambda,\Vxi})=0$. Moreover,
$$
\begin{aligned}
\frac{\partial G_{\lambda}}{\partial \ph} (\Vxi,\pl)[v] &  := v -\pi^\perp_{\lambda,\Vxi} (-\Delta)^{-\frac{1}{2}} f_\lambda'(\w + \pl)\pi_{\lambda,\Vxi}^\perp v  \\
& = \pi_{\lambda,\Vxi} v  + \pi_{\lambda,\Vxi}^\perp  \left[ \pi_{\lambda,\Vxi}^\perp v -(-\Delta)^{-\frac{1}{2}} \left(f_\lambda'(\w +\pl)\pi_{\lambda,\Vxi}^\perp v\right) \right] \\
& = \pi_{\lambda,\Vxi} v  + \pi_{\lambda,\Vxi}^\perp   L_{\lambda,\Vxi} (\pi_{\lambda,\Vxi}^\perp  v) - \pi_{\lambda,\Vxi}^\perp  \left[(-\Delta)^{-\frac{1}{2}} \left(\big(f_\lambda'(\w +\pl)-f_\lambda'(\w)\big) \pi_{\lambda,\Vxi}^\perp v\right) \right].\\
\end{aligned}
$$
For any $p\in (1,2)$ and $s>p$ such that $\frac{1}{p}+\frac{1}{s}>1$, Remark \ref{similar} gives
$$
\|f_\lambda'(\w +\pl)-f_\lambda'(\w)\|_{L^p(I)} =O( \lambda^{\frac{1-s}{s}}\|\ph_{\lambda,\Vxi}\|)=O(\lambda^{\frac{1}{p}+\frac{1}{s}-1}|\log \lambda|).
$$
Hence, using Sobolev's inequality, we can find $\alpha>0$ such that  
$$
\left\|\pi_{\lambda,\Vxi}^\perp  \left[(-\Delta)^{-\frac{1}{2}} \left(\big(f_\lambda'(\w +\pl)-f_\lambda'(\w)\big) \pi_{\lambda,\Vxi}^\perp v \right)  \right]\right\|= O(\lambda^\alpha\|v\|). 
$$
Then, we have 
$$\begin{aligned}
\left\|\frac{\partial G_{\lambda}}{\partial \ph} (\Vxi,\pl)[v] \right\| & \ge  \| \pi_{\lambda,\Vxi} v  + \pi_{\lambda,\Vxi}^\perp  L_{\lambda,\Vxi} (\pi_{\lambda,\Vxi}^\perp  v) \|+ O(\lambda^\alpha\|v\|)  \\
& \ge \frac{1}{\sqrt{2}}\| \pi_{\lambda,\Vxi} v \| + \frac{1}{\sqrt{2}} \|\pi_{\lambda,\Vxi}^\perp   L_{\lambda,\Vxi} (\pi_{\lambda,\Vxi}^\perp  v) \|+ O(\lambda^\alpha\|v\|) \\
& \ge \frac{1}{\sqrt{2}}\|\pi_{\lambda,\Vxi}v  \| + \frac{c}{{\sqrt{2}}} |\log \lambda|^{-1}\|\pi^\perp_{\lambda,\Vxi}v  \| + O(\lambda^\alpha\|v\|)\\
 & \ge  c \left( \|\pi_{\lambda,\Vxi}v  \| +\|\pi^\perp_{\lambda,\Vxi}v  \|\right) + O(\lambda^\alpha\|v\|)\\
& \ge  (c + O(\lambda^\alpha))\|v\|. 
\end{aligned}
$$
This implies that $\frac{\partial G_{\lambda}}{\partial \ph} (\Vxi,\pl)[v]$ is invertible. Then, the implicit function theorem gives that  $\ph_{\lambda,\Vxi}$ is of class $C^1$. 
\end{proof}

\section{Choice of the concentration points}\label{Sec:Points}
{Let $\ph_{\lambda,\xi}$ be as in Proposition \ref{PropPhi}. It is clear that if we find $\Vxi=(\xi_1,...,\xi_k)$ (depending on $\lambda$) such that \begin{equation}\label{ci=0}
c_i(\lambda,\Vxi)=0,\ \forall i=1,...,k,
\end{equation} 
then the function $u_{\lambda}:=\w+\ph_{\lambda,\xi}$ is solution for our initial problem \eqref{Eq}. In this section, we will prove that \eqref{ci=0} is satisfied when $\Vxi$ is a critical point of the reduced energy functional 
\begin{equation}\label{reducedF}
\F_\lambda(\Vxi) := J_{\lambda}(\w +\ph_{\lambda,\Vxi}),
\end{equation}
where 
$$
J_{\lambda}(u):= \frac{1}{2} \|u\|^2 -  \int_{I} g_\lambda(u) dx,\quad \text{ with } \quad g_\lambda(t):= \int_0^t f_\lambda(s) ds.
$$ 
In order to prove this, we will need the following preliminary estimate.

\begin{lemma}\label{Lemma normw} Let $F_1,\ldots,F_k$ be as in \eqref{F_i}. As $\lambda \to 0$, we have
\begin{equation}\label{normw}
\|\w\|^2 = -4\pi k \log \lambda -2\pi  \sum_{i=1}^k F_i(\Vxi) + O(\lambda|\log \lambda|),
\end{equation}
uniformly with respect to $\Vxi \in \mathcal P_{k,\eta}$. 
\end{lemma}
\begin{proof}
By definition of $\w$, in order to prove \eqref{normw}, it is sufficient to show that 
\begin{equation}\label{PUii}
\|PU_{\de_i,\Vxi_i}\|^2=  -4\pi \log \lambda -2\pi F_i(\Vxi) -4\pi^2 a_i \sum_{j\neq i} a_j G_{\xi_j}(\xi_i)+O(\lambda|\log \lambda|), \quad \text{ for } i =1,\ldots,k,
\end{equation}
and 
\begin{equation}\label{PUij}
<PU_{\de_i,\Vxi_i},PU_{\de_j,\Vxi_j}> = 4\pi^2 G_{\xi_j}(\xi_i) + O(\lambda|\log \lambda|), \quad \text{ for } i,j = 1,\ldots,k \text{ with } i\neq j. 
\end{equation}
Let us prove \eqref{PUii} first. For $i=1,\ldots,k$, since $e^{U_{\de_i,\Vxi_i}}=O(\lambda)$ in $\R\setminus (\xi_i-\frac{\eta}{2},\xi_i+\frac{\eta}{2})$, $PU_{\de_i,\Vxi_i}=0$ in $\R\setminus I$   and, by Lemma \ref{PU},  $\|PU_{\de_i,\Vxi_i}\|_{L^\infty(I)}=O(|\log \lambda|)$, we have  
$$
\|PU_{\de_i,\xi_i}\|^2 = \int_{\R} PU_{\de_i,\xi_i} e^{U_{\de_i,\xi_i}} dx = \int_{\xi_i-\frac{\eta}{2}}^{\xi_i+\frac{\eta}{2}}  PU_{\de_i,\xi_i} e^{U_{\de_i,\xi_i}} dx + O(\lambda |\log \lambda|). 
$$
Moreover, thanks to the estimates
$$
\int_{\xi_i-\frac{\eta}{2}}^{\xi_i+\frac{\eta}{2}} e^{U_{\de_i,\xi_i}}  dx  = 2\pi+O(\lambda) \quad \text{ and } \quad   \int_{\xi_i-\frac{\eta}{2}}^{\xi_i+\frac{\eta}{2}} U_{\de_i,\xi_i} e^{U_{\de_i,\xi_i}}  dx = -2\pi \log( 2\de_i) + O(\lambda |\log \lambda|),
$$ 
the expansion of $PU_{\de_i,\xi_i}$ from Lemma \ref{PU} yields
$$\begin{aligned}
\int_{\xi_i-\frac{\eta}{2}}^{\xi_i+\frac{\eta}{2}}  PU_{\de_i,\xi_i} e^{U_{\de_i,\xi_i}} dx  &  =  \int_{\xi_i-\frac{\eta}{2}}^{\xi_i+\frac{\eta}{2}}  (-\log (2\de_i) + U_{\de_i,\xi_i} + 2\pi H(\xi_i,x) )  e^{U_{\de_i,\xi_i}}   dx  +O(\lambda^2) \\
&  =-4\pi \log (2\de_i) +  4\pi^2 H(\xi,\xi_i) +O(\lambda|\log
 \lambda |)\\
 &  =-4\pi \log (2\de_i) + 2\pi F_i(\Vxi) - 4\pi^2 \sum_{j\neq i} G_{\xi_j}(\xi_i)  +O(\lambda|\log \lambda |).
\end{aligned}$$
Recalling that $\de_i$ is chosen as in \eqref{delta}, we have  $\log (2\delta_i) =\log \lambda + F_i(\Vxi)$, and we obtain  \eqref{PUii}.

With similar arguments, for $i\neq j$ we get
$$
\begin{aligned}
\int_{\R} (-\Delta)^\frac{1}{4} P U_{\de_i,\xi_i} (-\Delta)^\frac{1}{4} P U_{\de_j,\xi_j} dx &  =  \int_{\R} e^{U_{\de_i,\xi_i}} P U_{\de_j,\xi_j} dx \\
& = \int_{\xi_i-\frac{\eta}{2}}^{\xi_i+\frac{\eta}{2}}  e^{U_{\de_i,\xi_i}} P U_{\de_j,\xi_j} dx   + O(\lambda|\log \lambda|)\\
& = 2\pi \int_{\xi_i-\frac{\eta}{2}}^{\xi_i+\frac{\eta}{2}}  e^{U_{\de_i,\xi_i}} G_{\xi_j} dx  + O(\lambda|\log \lambda|) \\
& = 4\pi^2 G_{\xi_j}(\xi_i)  + O(\lambda|\log \lambda|),
\end{aligned}
$$
so that \eqref{PUij}  holds. 
\end{proof}

\begin{prop}\label{reduction}
For $\lambda \in (0,\lambda_0)$ and $\Vxi\in \mathcal P_{k,\eta}$,  the following conditions are equivalent:
\begin{enumerate}
\item $c_i(\lambda,\Vxi) = 0$, for $i=1,\ldots,k$.
\item $\nabla \F_\lambda (\Vxi) =0 $. 
\end{enumerate}
\end{prop}
\begin{proof}
By definition of $J_\lambda$, we have that 
$$
\nabla J_{\lambda}(u) = u - (-\Delta)^{-\frac{1}{2}} f_{\lambda}(\w),
$$
for any $u\in X^\frac{1}{2}_0(I)$. Then, recalling that $\ph_{\lambda}$ satisfies \eqref{EqPhi}-\eqref{EqSum}, we get
$$
\begin{aligned}
\nabla J_{\lambda}(\w+\ph_{\lambda,\Vxi}) & = \w+\ph_{\lambda,\Vxi} -  (-\Delta)^{-\frac{1}{2}} f_{\lambda}(\w+\ph_{\lambda,\Vxi}) = \sum_{j=1}^k c_j PZ_{1,j},
\end{aligned}
$$
For $i=1,\ldots,k$, by the chain rule, we find 
$$
\frac{\partial \F_{\lambda}}{\partial\xi_i} (\Vxi)   = <\nabla J_\lambda (\w+\ph_{\lambda}), \frac{d}{d \xi_i}\w +\frac{d}{d\xi_i} \ph_{\lambda,\Vxi} >  = \sum_{j=1}^k \beta_{ij} c_j,
$$
where
$$\beta_{ij} = <PZ_{1,j}, \frac{d}{d \xi_i}\w> +<PZ_{1,j},  \frac{d}{d\xi_i} \ph_{\lambda,\Vxi}>.
$$
Then, it suffices to show that the matrix $(\beta_{ij})$ is invertible. Indeed, this gives 
$$
\nabla \F_\lambda(\xi) = 0 \quad \Longleftrightarrow \quad  c_i(\lambda,\Vxi) =0, \quad i=1,\ldots,k.
$$
Let us then estimate the coefficients $\beta_{ij}$. First, for $i,h=1,\ldots,k$, we observe that 
$$\begin{aligned}
\frac{d}{d \xi_i} e^{U_{\de_h,\xi_h}} &  = \frac{\delta_{i,h}}{\de_i}  e^{U_{\de_i,\xi_i}} Z_{1,i}- \frac{1}{\de_h} e^{U_{\de_h,\xi_h}} Z_{0,h} \frac{\partial \de_h}{\partial \xi_i}  \\
& = \frac{\delta_{i,h}}{\de_i}  e^{U_{\de_i,\xi_i}} Z_{1,i}- e^{U_{\de_h,\xi_h}} Z_{0,h}  \frac{ \partial F_{h}}{\partial \xi_i}(\Vxi),
\end{aligned}
$$
where $\delta_{i,h}$ denotes the  Kronecker delta and we have used that $\Vde = (\de_1,\ldots,\de_k)$ is given by \eqref{delta}. Consequently 
$$
\frac{d}{d \xi_i} \w   =  \sum_{h=1}^k a_h \frac{d}{d\xi_i}  PU_{\de_h,\xi_h}  =  \frac{a_i}{\delta_i} PZ_{1,i} - \sum_{h=1}^k a_h PZ_{0,h} \frac{\partial F_h }{{\partial \xi_i}}(\Vxi) .
$$
Then, we infer 
\begin{equation}\label{PZDw}
 <PZ_{1,j},\frac{d \w}{d \xi_i}>  =   \frac{a_i}{\delta_i}<PZ_{1,j},PZ_{1,i}> - \sum_{h=1}^k a_h \frac{\partial F_{h}}{{\partial \xi_i}}(\Vxi)  <PZ_{1,j},PZ_{0,h}> = \frac{\pi a_i}{\de_i} \delta_{i,j} + O(1).  
\end{equation} 
Now, for $i,j=1,\ldots,k$, observe  that
$$\begin{aligned}
\pl \in K_{\Vde,\Vxi}^\perp \quad  &\Longrightarrow\quad  <PZ_{1,j},\pl>
 =0,\\
& \Longrightarrow\quad <\frac{d}{d\xi_i} PZ_{1,j},  \ph_{\lambda,\Vxi} > + <PZ_{1,j},  \frac{d}{d \xi_i} \pl> = 0.
\end{aligned}
$$
Note further that we have the identity 
$$
\begin{aligned}
\frac{d}{d\xi_i} e^{U_{\de_j,\xi_j}}Z_{1,j} %&= \delta_{i,j} \left( \frac{1}{\de_i} e^{U_{\de_i,\xi_i}}Z_{1,i}^2 -e^{2U_{\de_i,\xi_i}}Z_{0,i} \right)- \frac{2}{\delta_j} e^{U_{\de_j,\xi_j}}Z_{1,j}Z_{0,j} \frac{{\partial \delta_j}}{\partial \xi_i }\\
 &= \delta_{i,j} \left( \frac{1}{\de_i} e^{U_{\de_i,\xi_i}}Z_{1,i}^2 -e^{2U_{\de_i,\xi_i}}Z_{0,i} \right)- 2 e^{U_{\de_j,\xi_j}}Z_{1,j}Z_{0,j} \frac{{\partial F_{j}(\Vxi)}}{\partial \xi_i }  \\
& = O(\frac{1}{\de_i} e^{3 U_{\de_i,\xi_i}}|x-\xi_i|^2) + O(e^{2 U_{\de_i,\xi_i}})+ O(e^{2 U_{\de_j,\xi_j}}),
\end{aligned}
$$
where we have used that $|Z_{0,i}|,\ |Z_{0,j}|\leq 1$.

Then, since $\frac{d}{d\xi_i} PZ_{1,j} = (-\Delta)^{-\frac{1}{2}} \frac{d}{d\xi_i} e^{U_{\de_j,\xi_j}}Z_{1,j}$, by \eqref{EllEst} and \eqref{gen Est} we get that 
$$
\|\frac{d}{d\xi_i} PZ_{1,j}\| = O\left(\|\frac{d}{d\xi_i} e^{U_{\de_j,\xi_j}}Z_{1,j} \|_{L^2(I)}\right) = O(\lambda^{-\frac{3}{2}}).
$$
In particular, recalling that for $p\in (1,2)$ we have $\|\ph_{\lambda,\Vxi}\|=O(\lambda^\frac{1}{p} |\log \lambda|)$, we get 
\begin{equation}\label{PZDph}
\begin{aligned}
<PZ_{1,j}, \frac{d}{d\xi_i} \ph_{\lambda,\Vxi}> = - <\frac{d}{d\xi_i} PZ_{1,j},  \ph_{\lambda,\xi_i} > = O(\lambda^{-\frac{3}{2}}\|\pl\|) = O(\lambda^{\frac{1}{p}-\frac{3}{2}}|\log \lambda| ) = o(\lambda^{-1}),
\end{aligned}
\end{equation}
uniformly for $\Vxi \in \mathcal P_{k,\eta}$. For $\lambda$ small enough, using \eqref{delta}, \eqref{PZDw} and \eqref{PZDph}, we conclude that the matrix $(\beta_{ij})$ is dominant diagonal and thus invertible. This concludes the proof.  
\end{proof}

The following lemma describes the asymptotic behavior of $\F_\lambda$ as $\lambda \to 0$. 

\begin{lemma}\label{Lemma conv}
We have 
$$
\F_\lambda(\Vxi) = -2\pi k \log \lambda -2\pi k  - \pi \sum_{i=1}^k F_i(\Vxi) + o(1),
$$
where $o(1)\to 0$ as $\lambda\to 0$, uniformly for $\Vxi \in \mathcal{P}_{k,\eta}$. 
\end{lemma}
\begin{proof}
According to Lemma \ref{Lemma normw}, we have $\|\w\|^2= O(|\log \lambda|)$, so that
$$\begin{aligned}
\|\w+ \ph_{\lambda,\Vxi}\|^2 & = \|\w\|^2+ O(\|\w\|\|\ph_{\lambda,\xi}\|) + O(\|\ph_{\lambda,\Vxi}\|^2)\\
&  = \|\w\|^2 + O(\lambda^\frac{1}{p}|\log \lambda|^\frac{3}{2}).
\end{aligned}
$$
Noting that $g_\lambda = f_\lambda'-2\lambda$, by Remark \ref{similar}, for any $p\in (1,2)$ and $s>p$ such that $\frac{1}{s}+\frac{1}{p}>1$, one has
$$
\|g_\lambda(\w+\ph_{\lambda,\Vxi}) - g_\lambda(\w)\|_{L^1(I)} \le C \lambda^\frac{1-s}{s} \|\ph_{\lambda,\Vxi}\| = O(\lambda^{\frac{1}{p}+\frac{1}{s}-1} |\log \lambda|) =o(1), 
$$
as $\lambda\to 0$. Thus 
\begin{equation}\label{almost}
J_\lambda(\w+\ph_{\lambda,\Vxi}) = J_{\lambda}(\w) +o(1) = \frac{1}{2}\|\w\|^2 - \int_{I} g_\lambda(\w) dx  + o(1).
\end{equation}
Using again that $g_\lambda = f_\lambda'-2\lambda $ together with Lemma \ref{ExpDer} and \eqref{gen Est}, we find 
$$
\int_{I} g_\lambda(u) dx = \sum_{i=1}^k \int_{\xi_i-\frac{\eta}{2}}^{\xi_i+\frac{\eta}{2}} e^{U_{\de_i,\Vxi_i}} dx +    O(\lambda|\log \lambda|) = 2\pi k +O(\lambda|\log \lambda|).
$$ 
Then the conclusion follows by Lemma \ref{Lemma normw} and \eqref{almost}.
\end{proof}

The previous lemma shows that, up to constant terms that do not depend on $\Vxi$, the functional $\F_\lambda$ converges uniformly to a multiple of the function  
\begin{equation}\label{FirstF}
\F(\Vxi) := \frac{1}{2\pi} \sum_{i=1}^k {F_i (\Vxi)}  = \sum_{i=1}^k H(\xi_i,\xi_j) +\sum_{i,j=1, i\neq j}^k a_i a_j G(\xi_i,\xi_j).
\end{equation}
In the next section, we shall study the properties of $\F$ and exploit them to show that $\F_\lambda$ has a critical point  (a local minimum) in $\mathcal P_{k,\eta}$, provided $\eta$ is fixed small enough and the $a_i's$ have alternating sign.

%In the next section, we shall study the properties of $\F$ and we will prove that $\F$ has a maximum point $\bar \Vxi$ in $\mathcal P_{k,\eta}$ such that
%$$
%\max_{ \partial \mathcal P_{k,\eta}} \F < \F(\bar \Vxi),
%$$ 
%provided $\eta$ is fixed small enough and the $a_i's$ have alternating sign. We can deduce from this that $\F_\lambda$ has a critical point  (a local minimum) in $\mathcal P_{k,\eta}$.   

%If we prove that $\F$ has a critical point in $\mathcal{P}_{k,\eta}$ which is stable with respect to small uniform perturbations, then we can find a critical point of $\F_\lambda$ in $\mathcal P_{k,\eta}$ and we can conclude the proof of Theorem \ref{TrmBetter} thanks to Propositions \ref{reduction} and \ref{PropPhi}. 
}

{
\subsection{Existence of {a critical} point}
Let us now assume $a_i=-a_{i+1},$ $\forall i\in\{1,..,k-1\}$. {We refer to the appendix for some considerations concerning different possible choices of the $a_i's$. } With this assumption, the function $\F$ defined in \eqref{FirstF} becomes
\[\begin{split}
\F(\Vxi)= &\sum_{i=1}^kH(\xi_i,\xi_i)+ \sum_{i,j=1, i\neq j}^k (-1)^{i+j}G(\xi_i,\xi_j )\\
=&\sum_{i=1}^k\frac{1}{\pi}\log\left( 2(1-\xi_i^2)\right)) +\sum_{i,j=1, i\neq j}^k(-1)^{i+j}\frac{1}{\pi}\log\left(\frac{1-\xi_i\xi_j+\sqrt{(1-\xi_i^2)(1-\xi_j^2)}}{ |\xi_i-\xi_j|}\right).
\end{split}
\]

{The goal of this section is to show that the set of maximum points for  $\F$  on the set 
\begin{equation}\label{P_k}
\mathcal P_{k}:=\{ {\Vxi}  =(\xi_1,\ldots,\xi_k),\ -1<\xi_i<\xi_{i+1}<1,\,\forall i =1,\ldots, k-1 \},
\end{equation}
is a non-empty compact subset of $\mathcal P_{k}$, independently of the value of $k\in\N$. Combining this with Lemma \ref{Lemma conv}, we will prove that the functional $\F_\lambda$ defined in \eqref{reducedF} has a critical point in $\mathcal P_k$ (in fact in $\mathcal P_{k,\eta}$, if $\eta$ is small enough).}   The proof of this result is inspired by the the proof of Theorem $3.3$ in \cite{Bapiwe}. We will provide some details here, since having the explicit expression for the Green function of our operator simplifies considerably many  steps of the proof. For example, we easily get the following properties. 

\begin{lemma}\label{PropGr}
The following properties hold:
\begin{enumerate}[label={(\roman*)}]
\item \label{Step_I}   $H(\xi,\xi)\rightarrow -\infty$, as  $\xi \rightarrow \partial I=\{-1,1\} $
\item\label{Step_II}  Let $\eps>0$ there exists a constant $c(\epsilon)$ such that
$|H(\xi,\xi)|\leq c(\eps)$,  {if } $\dist(\xi,\partial I )\geq \eps $
\item\label{Step_III}  Let $\eps>0$ be small enough, there exist a constant $c(\epsilon)$ such that $|G(x,y)|\leq c(\eps)$,  {if } $|x-y|\ge \eps. $
\vspace{-0.3cm}
\item\label{StepIV} {For any $x,\xi\in I$, $x\neq \xi$, we have $\frac{d}{dx}G(\xi,x)= - \frac{1}{\pi} \frac{\sqrt{1-\xi^2}}{(x-\xi) \sqrt{1-x^2}}$}
\item\label{Step_IV} Given any three points $x<y<z\in I$,  we have  $G(x,z)-G(x,y)\leq 0.$
\end{enumerate}
\end{lemma}

%
%\begin{lemma}\label{Step_IV}
%Given any three points $x<y<z\in I$, the Green function satisfies 
%$$G(x,z)-G(x,y)\leq 0  .$$
%\end{lemma}
%\begin{proof}
%This follows directly from \ref{StepIV} of Lemma \ref{PropGr}, which shows that $G(x,\cdot)$ is monotone decreasing in the interval $(x,+\infty)$. 
%\end{proof}

The following lemma  provides upper bounds on $\F$. 

\begin{lemma}\label{PropLess} 
For any $\Vxi = (\xi_1,\ldots,\xi_k)\in \mathcal P_k$, we have
$$\sum_{i,j=1, i\neq j}^k(-1)^{i+j} G(\xi_i,\xi_j) \leq 0,$$
\end{lemma}
\begin{proof}
For any $1\le i\le k-1$, we set $$G_i(\Vxi):= \sum_{j=i+1}^k (-1)^{i+j} G(\xi_i,\xi_j),$$ so that 
$$
\sum_{i,j=1, i\neq j}^k(-1)^{i+j} G(\xi_i,\xi_j) = 2 \sum_{i=1}^{k-1} G_i(\Vxi).
$$
Then, it is sufficient to observe that $G_i(\Vxi)\le 0$ for any $1\le i\le k-1.$ Indeed, if $k-i$ is even, we have that 
$$
G_i(\Vxi)=\sum_{j=1}^{k-i} (-1)^j G(\xi_i,\xi_{i+j}) =  \sum_{j=2, \, j\, even}^{k-i} G(\xi_i,\xi_{i+j}) - G(\xi_i,\xi_{i+j-1}) \le 0, 
$$
where the last inequality follows by \emph{\ref{Step_IV}} of Lemma \ref{PropGr}. If instead $k-i$ is odd, then we have 
$$
G_i(\Vxi)=\sum_{j=1}^{k-i} (-1)^j G(\xi_i,\xi_{i+j}) = - G(\xi_i,\xi_{k}) + \sum_{j=2, \, j\, even}^{k-i-1} G(\xi_i,\xi_{i+j}) - G(\xi_i,\xi_{i+j-1}) \le 0, 
$$
where we used again property \emph{\ref{Step_IV}} of Lemma \ref{PropGr} together with the inequality $G(\xi_i,\xi_k)\ge 0$. 
\end{proof}

\begin{prop}\label{Step_V}
For any $k\in\N$, we have $\F(\Vxi)\rightarrow -\infty$ when $\dist\left(\Vxi,\ \mathcal \partial  \mathcal P_{k}\right)\rightarrow 0$. In particular, $\F$  has a maximum point in $\mathcal P_{k}$. Moreover, the set $M_{\F}$ of global maxima for $\F$ in  $\mathcal P_k$ is compact.
\end{prop}
\begin{proof}
{It is sufficient to show that, for any sequence $\Vxi^n=(\xi_1^n,\ldots,\xi_k^n)\in \mathcal P_{k}$ with $\dist(\Vxi^n,\partial \mathcal P_{k})\to 0$ as $n\to +\infty$, up to extracting a subsequence, one has $\F(\Vxi^n)\to -\infty$ as $n\to +\infty$}. If there exists $i\in \{1,\ldots,k\}$ such that $\xi_i^n \to \partial I $, then \emph{\ref{Step_I}} of Lemma \ref{PropGr} implies  that $H(\xi_i^n,\xi_i^n)\to -\infty$ and, thanks to Lemma \ref{PropLess},
$$
\F(\Vxi^n) \le \sum_{j=1}^k H(\xi_j^n,\xi_j^n) \le H(\xi_i^n,\xi_i^n) + \frac{k-1}{\pi} \log 2 \to -\infty,
$$
as $n\to +\infty$. Thus, up to a subsequence, we may assume that $\exists\;\eps >0$ such that $|\xi_i^n|\le 1-\eps$, $1\le i\le k$. Then $d(\Vxi^n, \partial \mathcal P_k)\to 0$ implies $\xi_{i}^n-\xi_{i+1}^n \to 0$ for some $1\le i\le k-1$.   Let $i_0$ be the maximal $i\in \{1,\ldots,k-1\}$ such that this property holds. Then, up to extracting another subsequence we may assume $\eps\le \xi_{i+1}-\xi_i \le 2$, for any $i_0<i\le k-1$.  Note  that \emph{\ref{Step_II}} and \emph{\ref{Step_III}} of Lemma \ref{PropGr} give
$$\begin{aligned}
\F(\Vxi^n) & = - \frac{1}{\pi}\sum_{i\neq j} (-1)^{i+j} \log |\xi_i^n-\xi_{j}^n| +O(1)\\
%& = -\frac{2}{\pi} \sum_{i=1}^{k-1} \sum_{j=i+1}^k  (-1)^{i+j} \log |\xi_i^n-\xi_{j}^n| + O(1) \\
& = -\frac{2}{\pi} \sum_{i=1}^{k-1} \sum_{j=1}^{k-i}  (-1)^{j} \log |\xi_i^n-\xi_{i+j}^n| + O(1) \\
& = -\frac{2}{\pi} \sum_{i=1}^{i_0} \underbrace{\sum_{j=1}^{k-i}  (-1)^{j} \log |\xi_i^n-\xi_{i+j}^n|}_{=:\zeta_i^n} + O(1).
\end{aligned}
$$
If $k-i$ is even, then we have
$$
e^{\zeta_i^n} = \prod_{j=1}^\frac{k-i}{2} \frac{|\xi_i^n-\xi_{i+2j}^n|}{|\xi_i^n-\xi_{i+2j-1}^n|} \ge 1,
$$
while if $k-i$ is odd, we have that 
$$e^{\zeta_i^n} = \frac{1}{|\xi_k^n-\xi_{i}^n|}\prod_{j=1}^\frac{k-i-1}{2} \frac{|\xi_i^n-\xi_{i+2j}^n|}{|\xi_i^n-\xi_{i+2j-1}^n|} \ge \frac{1}{2}.$$ Then, all the sequences $(\zeta_i^n)_{n\in \mathbb{N}}$, $1\le i\le k-1$ are bounded from below and we obtain 
$$
\begin{aligned}
\F(\Vxi_n) & \le -\frac{2}{\pi} \zeta_{i_0}^n + O(1)\\
& = -\frac{2}{\pi} \sum_{j=1}^{k-i_0} (-1)^j \log |\xi_{i_0}^n-\xi_{i_0+j}^n| + O(1)\\
& = \frac{2}{\pi} \log |\xi_{i_0}^n-\xi_{i_0+1}^n|+O(1) \to -\infty, 
\end{aligned}
$$
as $n\to \infty$. 
\end{proof}

%{
%\begin{rem}\label{TrMax}
%Proposition \ref{Step_V} implies that for any $k\in\N$, the functional $\F$  has a maximum point in $\mathcal P_{k}$. Moreover, the set $M_{\F}$ of global maxima for $\F$ in  $\mathcal P_k$ is a compact subset of $\mathcal P_k$.  {\color{red} (Maybe we can write this inside the previous proposition and not in a separate remark)}
%\end{rem}

\begin{cor}\label{final}
Let $\F_{\lambda}$ be as in \eqref{reducedF}. Then, there exists $\eta_0 \in (0,\frac{2}{k+1})$ such that $\F_\lambda$ has a critical point $\Vxi(\lambda) \in \mathcal P_{k,\eta_0},$  for any small $\lambda$. 
\end{cor}
\begin{proof}
By Proposition \ref{Step_V}, we can fix $\eta_0$ such that all the maxima of $\F$ in $\mathcal P_k$ belong to $\mathcal P_{k,\eta_0}$. In particular, since $\mathcal P_{k,\eta_0}$ is open, we have 
\begin{equation}\label{maxF}
\max_{\partial \mathcal P_{k,\eta_0}} \F  < \max_{ \overline{ \mathcal P}_{k,\eta_0}} \F.
\end{equation}
According to Lemma \ref{Lemma conv}, we have that  
$$
S_\lambda := \frac{1}{\pi^2} \left( 2\pi k \log \lambda +2\pi k -\F_\lambda \right)\to \F,
$$ 
uniformly in $\overline{\mathcal P}_{k,\eta_0}$ (in fact in $\mathcal P_{k,\eta}$, for any fixed $\eta <\eta_0$) as $\lambda \to 0$. Then, by \eqref{maxF}, we must have 
\[
\max_{\partial  \mathcal P_{k,\eta_0}} S_\lambda < \max_{\overline{\mathcal P}_{k,\eta_0}} S_\lambda,
\]
which implies that $S_\lambda$ has a maximum point $\Vxi(\lambda)$ in $\mathcal P_{k,\eta_0}$. In particular,  $\Vxi(\lambda)$ is a critical point for $S_{\lambda}$ and $\F_\lambda$. 
%Observe also that $\dist(\Vxi(\lambda),M_{\F}) \to 0$ as $\lambda \to 0$. Otherwise, we could find a sequence $\lambda_n \to 0$ such that $\dist(\Vxi,\lambda_n)\not \to 0$ as $n\to \infty$. Condition \eqref{MaxS} implies that $\Vxi$ is a maximum point for $S_\lambda$ on $\ov{\mathcal P}_{k,\eta_0}$, so that
%$$
%S_{\lambda_n} (\Vxi(\lambda_n)) = \max_{\overline{\mathcal P}_{k,\eta_0}} S_\lambda \to  \max_{\overline{\mathcal P}_{k,\eta_0}} \F.
%$$
%Up to a subsequence, we have that $\xi(\lambda_n)\to \bar \Vxi \in \ov{\mathcal P}_{k,\eta_0}.$ Then, the uniform convergence of $S_{\lambda}$ in $\ov{\mathcal P}_{k,\eta}$ implies that $S_{\lambda_n}(\Vxi(\lambda_n))\to \F(\bar \Vxi)$, so that  $\bar \Vxi$ is a maximum point for $\F$ in $\overline{\mathcal P}_{k,\eta_0}$. Since $\mathcal P_{k,\eta_0}$ contains $M_{\F}$, we conclude that $\bar \Vxi\in M(\F)$.  This contradicts the assumption $\dist(\Vxi,M_{\F})\not \to 0$. 
\end{proof}
%
%\begin{proof}
%First, we observe that Proposition \ref{PropLess} gives 
%$$
%\F(\Vxi)\le \frac{k\log 2}{\pi}, 
%$$
%so that $\F$ is bounded from above. Now, for any small $\eps>0$, set $\mathcal P_{k}^\eps:=\{ \Vxi \in \mathcal P_k \; :\; d(\Vxi,\partial \mathcal P_k)<\eps\}$. by Proposition \ref{Step_V}, we can choose $\eps$ small enough s.t. $\displaystyle{\F < \sup_{\mathcal P_{k}} \F -1}$ in $\mathcal P_{k}^\eps$. This implies that 
%\begin{equation}\label{Max}
%\sup_{\mathcal P_{k}} \F = \sup_{\mathcal P_k\setminus \mathcal P_k^\eps} \F. 
%\end{equation}
%Since $\mathcal P_k\setminus \mathcal P_k^\eps$ is a compact set,  $\F$ has a maximum in $\mathcal P_k\setminus \mathcal P_k^\eps$, which  is also a maximum in $\mathcal P_k$ because of \eqref{Max}. 
%\end{proof}

\begin{rem}
By construction, we also have that $\dist(\Vxi(\lambda),M_{\F}) \to 0$. In particular, for any sequence $\lambda_n \to 0 $, we have $\Vxi(\lambda_n) \to \bar \Vxi$ up to extracting a subsequence, where $\bar{\Vxi}$ is a maximum point for $\F$. 
%In the case $k=2$, we can prove that $\F$ has a unique maximum point $\bar \Vxi$ (cf. Appendix), so that $\Vxi(\lambda)\to \bar \Vxi$  as $\lambda\to 0$. 
\end{rem}
}

\section{Proof of the main Theorems}\label{Sec:Proofs}
We now collect the results of the previous sections to complete the proof of our main results. 

\medskip\medskip
\emph{Proof of Theorem \ref{TrmBetter}.} { Let $\w$ be as in \eqref{w}, with $\Vde = (\delta_1,\ldots,\de_k)$ as in \eqref{delta}. For a given $p\in (1,2)$, let $\lambda_0$ and $\ph_{\lambda,\Vxi}$ be as in Proposition \ref{PropPhi} and  let $\F_\lambda$ be as in \eqref{reducedF}. By Corollary \ref{final}, there exists  $\eta_0$ small enough such that $\F_\lambda$  has a critical point in $\mathcal P_{k,\eta_0}$. By Propositions \ref{reduction} and \ref{PropPhi}, setting $\ph_{\lambda}:= \ph_{\lambda,\Vxi(\lambda)}$ and  $\Vdelta(\lambda) := (\delta_1(\lambda,\Vxi(\lambda)),\ldots,\delta_k(\lambda,\Vxi({\lambda)}))$, we get that $u_\lambda:=\omega_{\Va,\Vdelta_\lambda,\Vxi({\lambda})} + \ph_\lambda $
is a solution of \eqref{Eq},  as claimed in Theorem \ref{TrmBetter}. Proposition \ref{PropPhi} also gives $\|\ph_\lambda\|=O(\lambda^\frac{1}{p}|\log \lambda|)\to 0$ as $\lambda \to 0$.  It remains to prove that $\|\ph_\lambda\|_{L^\infty(I)}\to 0$. Let us recall that $\ph_\lambda\in X_0^\frac{1}{2}(I)$ is a weak solution to
$$
(-\Delta)^\frac{1}{2}\ph_{\lambda} = \f'(\w) \ph_\lambda + E + N(\ph_\lambda)
$$
in $I$. Thanks to Lemma \ref{ExpDer}, we have 
$$
\int_{I} \left(\f'(\omega_{\Va, \Vde_\lambda,\Vxi(\lambda )}) |\ph_\lambda|\right)^p dx   = \sum_{i=1}^k \int_{\xi_i-\frac{\eta}{2}}^{\xi_i+\frac{\eta}{2}} O(e^{p U_{\de_i,\xi_i}}) |\ph_\lambda|^p dx  +o(1),
$$
as $\lambda\to 0$. But by H\"older's inequality (with any $q>1$) and \eqref{gen Est}, we get that 
$$\begin{aligned}
\int_{\xi_i-\frac{\eta}{2}}^{\xi_i+\frac{\eta}{2}} \left( e^{ U_{\de_i,\xi_i}} |\ph_\lambda|\right)^p dx  & \le  \left(\int_{\xi_i-\frac{\eta_0}{2}}^{\xi+\frac{\eta_0}{2}} e^{p q U_{\de_i,\xi_i}}dx \right)^\frac{1}{q}\left(\int_{I} |\ph_{\lambda}|^{\frac{pq}{q-1}}d x\right)^\frac{q-1}{q} \\ & = \left( O(\lambda^{1-qp }) \right)^\frac{1}{q} O(\|\ph_\lambda\|^p) \\ &= O(\lambda^{\frac{1}{q}+1 -p}|\log \lambda|^p ).
\end{aligned}$$
Since $p<2$, we can take $q>1$ such that $\frac{1}{q}+1-p>0$, so that $\|\f'(\w)\ph_\lambda\|_{L^p(I)}\to 0$ as $\lambda\to 0$. In addition,  Lemma \ref{Est_E} and Lemma \ref{est_L} give $\|E\|_{L^p(I)} \to 0$ and $\|N(\ph_\lambda)\|_{L^p(I)}\to 0$ as $\lambda\to 0$.  Thus  $(-\Delta)^\frac{1}{2} \ph_\lambda \to 0$ in $L^p(I)$ and elliptic estimates (see \cite[Theorem 13]{Tommaso}) imply $\|\ph_\lambda \|_{L^\infty(I)}\to 0$, as desired. \phantom{ } \hfill$\square$\medskip }

\medskip
We now turn to the proof of Theorem \ref{MainThm}. From now on, we let $u_\lambda$ be the solution constructed in Theorem \ref{TrmBetter}. {Note that $u_\lambda$ has the form  
$$
u_\lambda := \sum_{i=1}^{k}(-1)^{i-1} PU_{\de_i(\lambda),\xi_i(\lambda)} +\ph_\lambda,
$$
with $\Vxi(\lambda)=(\xi_1(\lambda),\ldots,\xi_k(\lambda)) \in \mathcal P_{k,\eta_0}$ for some small $\eta_0$, $ \Vdelta(\lambda)=(\de_1(\lambda),\ldots,\de_{k}(\lambda))$ such that  $\delta_{i}=O(\lambda)$ as $\lambda \to 0$, and $\ph_{\lambda}\in X_0^\frac{1}{2}(I)$ satisfying $\|\ph_\lambda\| +\|\ph_{\lambda}\|_{L^\infty(I)}\to 0$.  Up to extracting a subsequence, we may also assume that 
$$
\xi_{i}(\lambda) \to  \xi_{i}, \quad \text{ with } -1< \xi_1<\ldots< \xi_k <1.
$$
}

\begin{lemma}\label{SomeConvergence}
The following properties hold. 
\begin{itemize}
\item $u_\lambda$ blows-up with alternating sign at $\xi_1,\ldots,\xi_k$ as $\lambda\to 0$, that is \eqref{blow-up} holds for any small $\eps>0$. 
\item $u_\lambda\in C^\infty(I)$ and $u_\lambda \to u_0 :=\sum_{i=1}^k (-1)^{i-1} G_{\xi_1} $ in $C^\infty_{\loc}(I\setminus \{\xi_1,\ldots,\xi_k\})$ as $\lambda \to 0$. 
\item For any $\eps>0$, and $\alpha \in (0,{\frac{1}{2}})$, we have $$\left\| \frac{u_\lambda-u_0}{{\sqrt{d}}}  \right\|_{C^{0,\alpha}((-1,\xi_1-\eps))} + \left\| \frac{u_\lambda-u_0}{{\sqrt{d}}}  \right\|_{C^{0,\alpha}((\xi_{k}+\eps, 1))}   \to 0,$$ as $\lambda\to 0$, where $d(x):=1-|x|$ is the distance of $x$ from $\partial I$. 
\end{itemize}
\end{lemma}
\begin{proof}
In order to get the first property, it is sufficient to observe that Lemma \ref{PU} implies $PU_{\delta_i(\lambda),\xi_i(\lambda)}(\xi_i(\lambda))\to +\infty$ and $PU_{\de_j(\lambda),\xi_j(\lambda)}{(\xi_i(\lambda))} \to G_{\xi_j(\xi_i)}(\xi_i)$, for $j\neq i$. Since $\|\ph_\lambda\|_{L^\infty(I)}\to 0$ as $\lambda \to 0$, this gives the conclusion. 

Similarly, the second property follows by the boundedness of $u_\lambda$ in $L^\infty_{loc}(I\setminus \{\xi_1,\ldots,\xi_k\})$ and elliptic estimates for $(-\Delta)^\frac{1}{2}$ (see e.g. \cite{RosSer}).  

It remains to prove the third property.   {We focus first on the case $x\in[\xi_k+\epsilon,1)$ and we let $\psi$ be a smooth cut-off function such that $\psi\equiv 0$ on $(-\infty,\xi_{k}+\frac{\eps}{2})$  and  $\psi \equiv 1$ on $[\xi_k+\eps,\infty)$}. By construction, we have $u_\lambda \psi \equiv 0$ in $\R \setminus (\xi_k+\frac{\eps}{2},1)$. Moreover, for  $x\in (\xi_k+\frac{\eps}{2},1)$ we have
\begin{equation}\label{LapProduct}
(-\Delta)^\frac{1}{2} (u_\lambda \psi) = \psi (-\Delta)^\frac{1}{2} u_\lambda + u_\lambda (-\Delta)^{\frac{1}{2}}\psi  +  \int_{\R} \frac{(u_\lambda(x) - u_\lambda(y))(\psi(x)-\psi(y))}{|x-y|^2}dy.
\end{equation}

Note that the last integral is well defined since $  u_\lambda\in C^{0,\frac{1}{2}} (\R)$ and $\psi\in C^\infty(\R)$.  {Moreover  $\psi\in C^\infty(\R)\cap L^\infty(\R)$ implies  $(-\Delta)^\frac{1}{2}\psi\in C^\infty(\R)$ (see e.g. \cite[Proposition 2.1.4]{Silv}). Then, since} $u_\lambda$ is uniformly bounded in $L^\infty_{loc}(\R \setminus \{\xi_1,\ldots,\xi_k\})$, and $u_\lambda$ solves \eqref{Eq}, we have that 
$$
\|\psi (-\Delta)^\frac{1}{2} u_\lambda\|_{L^\infty({ (\xi_k+\frac{\eps}{2},1)})} +\|u_\lambda (-\Delta)^\frac{1}{2} \psi\|_{L^\infty({ (\xi_k+\frac{\eps}{2},1)})} \le C,
$$ 
for some $C>0$, depending only on $\eps$. Now, if $x\in [\xi_k+2\eps,1)$, then 
$$
\begin{aligned}
\left|\int_{\R} \frac{(u_\lambda(x) - u_\lambda(y))(\psi(x)-\psi(y))}{|x-y|^2}dy\right| & = \left|\int_{-\infty}^{\xi_k+\eps} \frac{(u_\lambda(x) - u_\lambda(y))(1-\psi(y))}{|x-y|^2}dy \right|\\
& \le \frac{1}{\eps^2} \left(\|{ u_\lambda}\|_{L^\infty(\xi_k+2\eps,1)}+\|{  u_\lambda}\|_{L^1(\R)}\right)\le C.
\end{aligned}
$$
If instead $x\in (\xi_k+\frac{\eps}{2},\xi_k+2\eps)$, then 
$$
\begin{aligned}
\int_{\R} \frac{(u_\lambda(x) - u_\lambda(y))(\psi(x)-\psi(y))}{|x-y|^2}dy &  = \int_{x-\frac{\eps}{3}}^{x+\frac{\eps}{3}} \frac{(u_\lambda(x) - u_\lambda(y))(\psi(x)-\psi(y))}{|x-y|^2}dy \\
& \quad + O\left(\|\psi\|_{L^\infty(\R)} (\|{ u_\lambda}\|_{L^1(\R)} +  \|{  u_\lambda}\|_{L^\infty((\xi_k+\frac{\eps}{2},1) )})\right),
\end{aligned}
$$
with 
$$
\left|\int_{x-\frac{\eps}{3}}^{x+\frac{\eps}{3}} \frac{(u_\lambda(x) - u_\lambda(y))(\psi(x)-\psi(y))}{|x-y|^2}dy \right|\le \frac{2\eps}{3} \|u_\lambda'\|_{L^\infty ({ \xi_k+\frac{\eps}{6},\xi_k+\frac{7}{3}\eps)}} \|\psi'\|_{L^\infty(\R)}.
$$

We have so proved that the RHS of \eqref{LapProduct} is bounded in $L^\infty((\xi_k+\frac{\eps}{2},1))$. Then, the regularity results of Ros-Oton and Serra { (Theorem $1.2$ in \cite{RosSer})} implies that $\|\frac{u_\lambda}{\sqrt{d}}\|_{C^{0,\beta}((\xi_k+\eps,1))}\le \|\frac{u_\lambda \psi}{\sqrt{d}} \|_{C^{0,\beta}((\xi_k+\frac{\eps}{2},1))}\le C$ for any $\beta<\frac{1}{2}$. In particular, we have that $\|\frac{u_\lambda-u_0}{\sqrt{d}}\|_{C^{0,\alpha}}\to 0$, for any $\alpha<\beta {<\frac{1}{2}}$. With similar arguments, we prove an analogue convergence result in $(-1,\xi_1-\eps)$. 

\end{proof}

%Throughout this section we shall denote $\xi_0 = -1$ and $\xi_{k+1}=1$. 
The main  step in the proof on Theorem \ref{MainThm} consists in showing that the limit profile $u_0$  has exactly $k-1$ zeros in $I \setminus (\xi_1,\ldots,\xi_k)$. In fact, we shall prove that $u_0$ is strictly monotone in each of the intervals $(\xi_i,\xi_{i+1})$. In the following it is useful to denote  $\xi_0 = -1$ and $\xi_{k+1}=1$. 

\begin{prop}\label{monotone}
For any given  $k\in \N$  and $\Vxi= (\xi_1,\ldots,\xi_k)$ with $-1=\xi_0<\xi_1<\ldots< \xi_k<\xi_{k+1}=1$, consider the function 
$$
u_0=\sum_{i=1}^k (-1)^{i-1} G_{\xi_i}.
$$
Then, there  exists a constant $c= c(k,\Vxi)$ such that, for any $ j=0,\ldots,k$, we have
$$
(-1)^{j}u_0'(x) \sqrt{1-x^2}\ge c> 0 \quad \text{ for } x\in (\xi_j,\xi_{j+1}).
$$
\end{prop}
\begin{proof}
Throughout the proof we denote $\alpha(x):=\sqrt{1-x^2}$, $x\in I$. 

\medskip
{\emph{Step 1}:} There exists $c_1= c_1(\xi_1,\ldots,\xi_k)$ such that $G_{\xi_i}'\alpha  \ge c_1$ in $(-1,\xi_i)$ and $G_{\xi_i}'\alpha \le -c_1$ in $(\xi_i,1)$. 

\medskip 
Fix $i\in\{1,\ldots,k\}$. According to \ref{StepIV} of Lemma \ref{PropGr}, we have that
$$
G_{\xi_i}'(x)= - \frac{1}{\pi} \frac{\sqrt{1-\xi_i^2}}{(x-\xi_i) \alpha(x)}.
$$
In particular 
$$
|G'_{\xi_i}(x)|\alpha(x) =  \frac{\sqrt{1-\xi_i^2}}{\pi |x-\xi_i|} \ge  \frac{\sqrt{1-\xi_i^2}}{2\pi} \ge { \frac{\displaystyle{\sqrt{1-\max_{1\le j\le k}|\xi_j|^2}}}{2\pi}},
$$
where we used $|x-\xi_i|\le 2$. Since $G_{\xi_i}'>0$ in $(\xi_0,\xi_i)$ and $G_{\xi_i}'<0$ in $(\xi_i,\xi_{k+1})$, the inequality above gives the conclusion. 

\medskip

{\emph{Step 2}}: Assume $k\ge 2$ and for any $1\le i\le k-1$ set $g_i:= G_{\xi_i}-G_{\xi_{i+1}}$. There exists a constant $c_2>0$ such that $g_i'\alpha \ge c_2$ in $(\xi_0,\xi_{i})\cup (\xi_{i+1},\xi_{k+1})$ and $g_i'\alpha \le -c_2$ in $(\xi_i,\xi_{i+1})$,  for any  $1\le i\le k$. 

\medskip
By Step 1, we know that $G_{\xi_i}'\alpha \le -c_1$ and $G_{\xi_{i+1}}'\alpha \ge c_1$ in  $(\xi_i,\xi_{i+1})$. This immediately gives $g_i' \alpha \le -2c_1$ in $(\xi_i,\xi_{i+1})$.   Let us now assume $x<\xi_i$ or $x>\xi_{i+1}$. As in Step 1, we have the explicit expression 
$$
\begin{aligned}
g_i'(x)\ge G_{\xi_i}'(x)-G_{\xi_{i+1}}'(x) & = \frac{1}{\pi} \frac{\sqrt{1-\xi_{i+1}^2}}{(x-\xi_{i+1}) \alpha(x)} - \frac{1}{\pi} \frac{\sqrt{1-\xi_i^2}}{(x-\xi_i) \alpha(x)}\\
& =  \frac{f_x(\xi_{i+1})-f_x(\xi_i)}{\pi \alpha(x)}, 
\end{aligned}
$$
where $f_x(t):= \frac{\sqrt{1-t^2}}{x-t}$. If $x<\xi_i$ or $x>\xi_{i+1}$, using that $f_x\in C^1((-1,1)\setminus \{x\})$, we get that 
$$
f_x(\xi_{i+1})-f_x(\xi_i) = f'_x(\bar \xi)(\xi_{i+1}-\xi_{i}).
$$
where $\bar \xi$ is a point between $\xi_i$ and $\xi_{i+1}$. In particular 
$$
f'_x(\bar \xi) = \frac{1-x \bar \xi}{(x-\bar \xi)^2\sqrt{1-\bar \xi^2}} \ge \frac{1-|\bar \xi|}{(x-\bar \xi)^2\sqrt{1- \bar \xi^2}}  \ge \frac{1 - \max\{|\xi_i|,|\xi_{i+1}|\}}{(x-\bar \xi)^2}\ge {\frac{1-M(\Vxi)}{4}},
$$
{where $M(\Vxi):= \displaystyle{\max_{1\le j\le k} |\xi_j|\in (0,1)}$.}  We can so conclude that 
$$
g_i'(x)\alpha(x ) \ge {\frac{(1-M(\Vxi))(\xi_{i+1}-\xi_i)}{4 \pi} \ge \frac{(1 -M(\Vxi))\sigma(\Vxi)}{4 \pi}},
$$
{where $\sigma(\Vxi):=\min_{1\le j\le k} \xi_{j+1} - \xi_{j} >0$}. The RHS is a constant depending only on $k$ and $\Vxi$. 
\medskip

{\emph{Step 3}}: Conclusion of the proof. 

\medskip
If $k=1$ or $k=2$, the conclusion follows directly from Step 1 and 2. 

Assume $k\ge 3$, $k$ odd. For $1\le i\le k-1$ let $\gi$ be as in Step 2. 
%
% As a consequence of Steps 1 and 2, we can find $c$ depending on $\xi_1,\ldots,\xi_k$ such that  
%\begin{equation}\label{DerC}\begin{aligned}
%{g_i}'\ge c  &\;\text{ in } (-1,\xi_i) \cup (\xi_{i+1},1),\;  1\le i\le k-1\\
%G_{\xi_i} \ge c & \;\text{ in } (-1,\xi_i),  \;1\le i\le k \\
%G_{\xi_i} \le -c &\; \text{ in } (\xi_i,1),\;  1\le i\le k
%\end{aligned}
%\end{equation}
We can write 
\begin{equation}{\label{Sum1}}
u_0 = \sum_{i=1, \, i\, odd}^{k-2} g_i + G_{\xi_{k}} 
\end{equation}
and 
\begin{equation}\label{Sum2}
u_0 = G_{\xi_1} - \sum_{i=2,i \, even}^{k-1} g_i. 
\end{equation}
Note that if $1\le i\le k-2$ is odd, and if $0\le j\le k-1$ is even, the interval $(\xi_{j},\xi_{j+1})$ is contained in  $(-1,\xi_i) \cup (\xi_{i+1},1)$ and in $(-1,\xi_k)$. In particular, Steps 1 and 2 guarantee the product of $\alpha$ with any of the functions appearing in \eqref{Sum1} is increasing in $(\xi_j,\xi_{j+1})$. In fact,  we get
$$
u_0' \,\alpha \ge  \frac{k-1}{2} c_2+c_1, \quad \text{ for any } j \text{ even}.
$$
Similarly, when $2F\le i\le k-1$  is even and $1\le j\le k$ is odd, then $(\xi_j,\xi_{j+1})$ is contained in $(-1,\xi_i)\cup (\xi_{i+1},1)$ and in $(\xi_1,1)$. Therefore, \eqref{Sum2} together with Steps 1 and 2 yields 
$$
-u_0'\,\alpha\ge c_1+ \frac{k-1}{2}c_2, \quad \text{ in } (\xi_j,\xi_{j+1}), \quad j \text{ odd}.  
$$

Finally, assume $k$  even and $k\ge 4$.  Then, we can decompose
\begin{equation}\label{Sum3}
u_0 = \sum_{i=1, i \,odd}^{k-1} g_i 
\end{equation}
and 
\begin{equation}\label{Sum4}
u_0= G_{\xi_1} - \sum_{i=2,i\, even}^{k-2} g_i -G_{\xi_k}.
\end{equation}
As before, for $0\le j\le k$ even, we have $(\xi_j,\xi_{j+1})\subseteq (-1,\xi_i)\cup (\xi_{i+1},1)$ and $g_i'\alpha \ge c_2$ in $(\xi_j,\xi_{j+1})$ (by Step 2), for any odd $1\le i\le k-1$. Then \eqref{Sum3} yields 
$$
u_0'\,\alpha \ge  \frac{k}{2} c_2, \quad \text{ in } (\xi_j,\xi_{j+1})\text{ for any } j \text{ even}.
$$ 

If instead $j$ is odd, one has $(\xi_j,\xi_{j+1})\subseteq (-1,\xi_i)\cup (\xi_{i+1},1)$ and $g_i'\,\alpha \ge c$ in $(\xi_j,\xi_{j+1})$, for any $2\le i\le {k-2}$ even. Moreover, since $(\xi_j,\xi_{j+1})\subseteq (\xi_1,1)\cap (-1,\xi_k)$, we also get $G_{\xi_1}' \alpha \le -c_1 $ and $G_{\xi_k}'\alpha \ge c_1$. Then, thanks to \eqref{Sum4} we find that 
$$
-u_0'\, \alpha \ge  2c_1+ \frac{k-2}{2} c_2, \quad \text{ in } (\xi_j,\xi_{j+1})\text{ for any } j \text{ odd }.
$$
\end{proof}
%\begin{rem}
%For any $x\in I$, we have the inequalities 
%\begin{equation}\label{RmkDist}
%d(x,\partial I)^\frac{1}{2}\le \sqrt{1-x^2} \le \sqrt{2} d(x,\partial I)^\frac{1}{2}.
%\end{equation}

%In particular, up to changing the value of $c$, we may rewrite the {\color{blue}above} proposition as 
%$$
%(-1)^{j}u_0'(x) d(x,\partial I)^\frac{1}{2} \ge c \quad \text{ in }( \xi_j,\xi_{j+1}) 
%$$
%\end{rem}

{ We can now complete the proof of Theorem \ref{MainThm}.}
\medskip

\medskip
\emph{Proof of Theorem \ref{MainThm}.}
Let $u_\lambda$ be the solution constructed in Theorem \ref{TrmBetter}. {In view of Lemma \ref{SomeConvergence}, we only need to} prove that, for $\lambda$ small enough, $u_\lambda$ has exactly $k$ nodal regions in $I$ or, equivalently, exactly $k-1$ zeroes in $I$. Let us fix $\eps>0$ small enough so that
\begin{equation}\label{sign}
(-1)^{i-i} u_0(\xi_i+\eps) >0, \quad (-1)^{i-i} u_0(\xi_i-\eps) >0 \quad \text{ and } \quad (\xi_i-2\eps,\xi_i+2\eps)\subseteq I \setminus \bigcup_{j\neq i} (\xi_j-\eps,\xi_j+\eps), 
\end{equation}
for $i=1,\ldots,k$. Let us split $I:= I_\eps^1 \cup I_\eps^2 \cup I_\eps^3$, where 
$$
I_\eps^1:= (-1,\xi_1-\eps]\cup [\xi_k+\eps ,1), \quad  I_\eps^2 := \bigcup_{i=1}^k (\xi_i-\eps,\xi_i+\eps), \quad I_\eps^3 = I \setminus (I_{\eps}^1 \cup I_\eps^2).
$$
First, we observe that $u_\lambda$ has no zeroes in $I_{\eps}^1$.
Using Proposition \ref{monotone}, in  $[\xi_k+\eps,1)$, we can write 
$$(-1)^{k-1}u_0(x) = \int_{x}^1 (-1)^k u_0'(t)dt \ge \int_{x}^1 \frac{c}{\sqrt{1-t^2}} dt \ge \frac{c}{\sqrt{2}} \int_x^1 \frac{1}{\sqrt{1-t}} dt  = c\sqrt{2(1-x)}  = \sqrt{2d(x)},$$
where {$d(x)=1-|x|=\dist(x,\partial I)$.} Similarly, for $x\in (-1,\xi_1-\eps]$, we can write 
$$u_0(x) = \int_{-1}^x u_0'(t)dt \ge \int_{x}^1 \frac{c}{\sqrt{1-t^2}} dt \ge \frac{c}{\sqrt{2}} \int_x^1 \frac{1}{\sqrt{1+t}} dt  = c\sqrt{2(1+x)}  = \sqrt{2d(x)}.$$
Thanks to Lemma \ref{SomeConvergence}, we get
$
|u_\lambda(x)| \ge  {\sqrt{d(x)}}
$
in $I_\eps^{1}$, provided $\lambda$ is sufficiently small. This shows that $u_\lambda$ has no zeroes in $I_{\eps}^1$. 
\medskip

Next, we observe that $u_\lambda$ has no zeros in $I_\eps^2$. {Let us fix $1\le i\le k$.} Lemma \ref{PU} gives that 
$$
\begin{aligned}
u_\lambda & = (-1)^{i-1}PU_{\de_i(\lambda),\xi_i(\lambda)} +  2\pi \sum_{j\neq i} (-1)^{j-1}G(\xi_j(\lambda),\xi_i) +O(|\cdot-\xi_i|) \\
& = (-1)^{i-1}\log \left(\frac{1}{\delta_i(\lambda)^2+|x-\xi_i(\lambda)|^2}\right) + 2\pi (-1)^{i-1}H(\xi_i,\xi_i)+  2\pi \sum_{j\neq i} (-1)^{j-1} G(\xi_j,\xi_i) +O(\eps)
\end{aligned}
$$
in $(\xi_i-\eps,\xi_i+\eps)$, if $\lambda$ is small enough. Moreover, we may assume that $|\xi_i(\lambda)-\xi_i|\le \eps$ and $\delta_i(\lambda)\le \eps$. In particular, we have that $|x-\xi_i(\lambda)|^2 +{\delta_i(\lambda)}^2\le 5\eps^2$  for any $x\in (\xi_i-\eps,\xi_i+\eps)$. Then, we get 
$$
|u_\lambda| \ge \log\frac{1}{5\eps^2} - O(1). 
$$
Thus, we have $|u_\lambda|\ge 1$  in $(\xi_i-\eps,\xi_i+\eps)$ if $\eps$ is fixed small enough. 

\medskip
Finally, let us consider the interval $I^3_\eps$. Note that $I^{3}_\eps$ has exactly $k-1$ connected components, namely we have 
$$
I^{3}_\eps = \bigsqcup_{i=1}^{k-1} J_{i,\eps}, \quad \text{ where } \quad J_{i,\eps}=[\xi_i+\eps,\xi_{i+1}-\eps].  
$$  
By \eqref{sign} and Proposition  \ref{monotone}, we know that for any $1\le i\le k-1$, if $\epsilon$ is small enough, we have 
$$
(-1)^{i} u_0(\xi_i+\eps) <0, \quad (-1)^{i} u_0(\xi_{i+1}-\eps) >0, \quad \text{ and } \quad (-1)^{i} u_0' \ge c \text{ in } J_{i,\eps}.  
$$
Since $u_\lambda \to u_0$ in $C^1(\bar I_{\eps}^3)$ by Lemma \ref{SomeConvergence}, this implies that 
$$
(-1)^{i} u_\lambda(\xi_i+\eps) <0, \quad (-1)^{i} u_\lambda(\xi_{i+1}-\eps) >0, \quad \text{ and } \quad (-1)^{i} u_\lambda' \ge c \text{ in } J_{i,\eps}.  
$$
Then $u_\lambda$ has exactly one zero in $J_{i,\eps}$ for any $1\le i\le k-1$. We can so conclude that $u_\lambda$ has exactly $k-1$ zeroes in $I^3_\eps$ (and thus in $I$), as claimed. 
\phantom{ } \hfill$\square$\medskip

\section*{Appendix: Some special cases}
In the proof of Theorem \ref{TrmBetter}, we had to assume that the coefficients $a_1,\ldots,a_k\in\{-1,1\}$ appearing in front of the bubbles $PU_{\de_i,\xi}$  in the expression of the approximate solution  $\w$ are sign-alternating i.e. $a_i = -a_{i+1}$ for $1\le i\le k-1$.  This condition has been used in order to ensure the existence of a maximum point for the functional 
$$
\F(\Vxi) = \sum_{i=1}^k H(\xi_i,\xi_i) + \sum_{i\neq j} a_i a_j G(\xi_i,\xi_j),
$$
in the set $\mathcal P_k$ defined in \eqref{P_k}, as well as the validity of Proposition  \ref{Step_V}.  It is simple to see that this strategy cannot be used for different choices of the $a_i's$.  In fact, if there exists $i\in\{1,\ldots,k\}$, such that $a_i = a_{i+1}$ then $\F$ is not bounded from above.

However, it is interesting to investigate whether one can find different kinds of critical points. Indeed, since it is possible to show that the convergence in Lemma \ref{Lemma conv} holds in the $C^1$-sense, we can construct solutions to \eqref{Eq} whenever we can find a $C^1$-stable critical point for $\F$.  A complete answer to this question can be given for $k=1$ or $k=2$, since one can explicitly find all the critical points of $\F$. In fact, we have the following:

\begin{itemize}
\item In the case $k=1$, we have
$$\F(\xi_1)=H(\xi_1,\xi_1)=\frac{1}{\pi}\log 2(1-\xi_1^2).$$ 
Then $\F$ does not depend on the choice of $a_1$ and has only one critical point at $\xi_1=0$ (a non-degenerate maximum point).

\item In the case $k=2$, we should find critical points of 
\begin{equation*}
\begin{split}
\F(\xi_1,\xi_2)&=  H(\xi_1,\xi_1)+ H(\xi_2,\xi_2)+2a_1 a_2 G(\xi_1,\xi_2)\\ &= \frac{1}{\pi}\log(1-\xi_1^2)(1-\xi_2^2) +\frac{2a_1a_2}{\pi}\log\frac{1-\xi_1\xi_2+\sqrt{(1-\xi_1^2)(1-\xi_2^2)}}{ |\xi_1-\xi_2|}.
\end{split}
\end{equation*}
This leads to two possible configurations:
\begin{itemize}
\item If we choose $a_1=-a_2$, we can easily see that $\F$ has only one critical point in $\mathcal P_2$, located at $(\xi_1,\xi_2)=(-\frac{1}{\sqrt{3}},\frac{1}{\sqrt{3}})$. This point is a non-degenerate global maximum. 
\item If we choose $a_1=a_2$,  we can easily see $\F$ has no critical points in $\mathcal P_2$.
\end{itemize}
\end{itemize}

We conjecture that for $k\ge 3$, the function $\F$ has a unique critical point (the global maximum) if the $a_i$'s have alternating sign, and has no critical point otherwise.

\bibliographystyle{abbrv}

\end{document}